\RequirePackage{ifpdf}
\ifpdf 
\documentclass[pdftex]{sigma}
\else
\documentclass{sigma}
\fi

\newcommand{\grad}{\nabla}

\newcommand{\mm}{\mathcal{M}}
\newcommand{\m}{\mathbf{M}}
\newcommand{\nn}{\mathcal{N}}

\newcommand{\dd}{\mathcal{D}}
\newcommand{\ff}{\mathcal{F}}

\newcommand{\hh}{\mathcal{H}}

\newcommand{\vv}{\mathcal{V}}

\newcommand{\h}{\mathfrak{h}}
\newcommand{\g}{\mathfrak{g}}
\newcommand{\dm}{\dd^1(\m)}
\newcommand{\vm}{\vv(\m)}
\newcommand{\fm}{\ff(\m)}

\begin{document}


\renewcommand{\PaperNumber}{109}

\FirstPageHeading

\ShortArticleName{Quasi-Exactly Solvable Schr\"odinger Operators in Three Dimensions}

\ArticleName{Quasi-Exactly Solvable Schr\"odinger Operators\\ in Three Dimensions}

\Author{M\'elisande FORTIN BOISVERT}

\AuthorNameForHeading{M. Fortin Boisvert}

\Address{Department of Mathematics and Statistics, McGill University, Montr\'eal, Canada, H3A~2K6}

\Email{\href{mailto:boisvert@math.mcgill.ca}{boisvert@math.mcgill.ca}}

\URLaddress{\url{http://www.math.mcgill.ca/boisvert/}}

\ArticleDates{Received October 01, 2007, in f\/inal form November 02,
2007; Published online November 21, 2007}

\Abstract{The main contribution of our paper is to give a partial
classif\/ication of the quasi-exactly solvable Lie algebras of f\/irst order dif\/ferential
operators in three variables, and to show how this can be applied to
the construction of new quasi-exactly solvable Schr\"odinger operators in three
dimensions.}

\Keywords{quasi-exact solvability; Schr\"odinger operators; Lie
algebras of f\/irst order dif\/fe\-ren\-tial operators; three dimensional
manifolds}

\Classification{81Q70; 22E70; 53C80}

\section{Introduction}

  Recall that a Schr\"odinger operator on a $n$-dimensional
 Riemannian manifold $(\mathbf{M},g)$ is a second order linear
 dif\/ferential operator  of the form
\[ \hh_0=-\frac{1}{2}\Delta+U,\] where $\Delta$ is the
Laplace--Beltrami operator and $U$ is the potential function for the
physical system under consideration. A question of fundamental
interest in quantum mechanics is to construct eigenfunctions $\psi$
of the Schr\"odinger operator $\hh_0$. One approach to this problem,
which is based on the representation theory of Lie algebras, is to
consider Schr\"odinger operators $\hh_0$ which are quasi-exactly solvable, in a
sense
that will be def\/ined below.

We begin by considering the case of a general linear second-order
dif\/ferential operator $\hh$, given in local coordinates by \[
\hh=\sum_{i,j=1}^nA^{ij}\partial_i
\partial_j+\sum_{i=1}^n B^i\partial_i+C. \]
The operator $\hh$ is said to be \emph{Lie algebraic} if it is an
element of the universal enveloping algebra of $\g$, a f\/inite
dimensional Lie algebra of f\/irst order dif\/ferential operators. More
explicitly,
\begin{equation}\label{liealge} \hh=\sum_{a,b=1}^mC_{ab}T^a T^b+\sum_{a=1}^m
C_aT^a+C_0,\end{equation} where
\begin{equation}\label{gen} T^a=v^a+\eta^a, \qquad  1\leq a \leq m,\end{equation}
is a basis or $\g$. In (\ref{gen}), the operators $v^1,\dots,v^m$ are
vector f\/ields, and $\eta_1,\dots,\eta_m$ are multiplication operators.
A Lie algebra $\g$ of f\/irst order dif\/ferential operators is said to
be {\it quasi-exactly solvable} if one can f\/ind explicitly a f\/inite
dimensional $\g$-module $\nn$ of smooth functions, i.e., if
$\nn=\{ h^1,\dots,h^r\}$ with $T^a(\nn)\subset\nn$ for all $1 \leq a
\leq m$. A Lie algebraic operator $\hh$, is said to be {\it
quasi-exactly solvable} if it lies in the universal enveloping
algebra of a quasi-exactly solvable Lie algebra of f\/irst order
dif\/ferential operators. Obviously, $\hh(\nn)\subset\nn$, i.e.\
the module $\nn$ will
 be f\/ixed by the operator $\hh$. Moreover, if the functions contained in the module $\nn$ are
square integrable with respect to the Riemannian measure
$\sqrt{g}dx^1\cdots dx^n$, where $g$ is the determinant of the covariant
metric, the operator $\hh$ is said to be a
{\it normalizable} quasi-exactly solvable operator.

We can see from the above def\/initions that the formal eigenvalue
problem for quasi-exactly solvable Schr\"odinger operators can be solved partially by
elementary linear algebraic methods. Indeed, the operator $\hh$ is
self-adjoint with respect to the inner product associated to the
standard measure, therefore the restriction of $\hh$ to the f\/inite
dimensional module $\nn$ is a Hermitian f\/inite dimensional linear
operator. Thus, one can in principle, compute $r=\dim (\nn)$
eigenva\-lues of $\hh$, counting multiplicities, by diagonalizing the
$r\times r$ matrix representing $\hh$ in a basis
of~$\nn$.

It seems that the concept of a ``spectrum generating algebra" was
f\/irst introduced by Goshen and Lipkin  in \cite{GL} in 1959.
However, this paper did not seem to have been be noticed by the
community and, ten years later, spectrum generating algebras were
independently rediscovered by two groups of physicists, see
\cite{BB} and \cite{DGN}. Their work was an impetus for further
research in this area as one can see by browsing in the two volume
set of reprints \cite{BNB} and the conference proceedings \cite{GO}.
A survey of the history and the contribution papers related to the
spectrum generating algebras is given in the review paper of B\"ohm
and Ne'eman, which appears at the beginning of \cite{BN}. In the
early $1980$'s, Iachello, Levine, Alhassid, G\"ursey and
collaborators exhibited applications of spectrum generating algebras
to molecular spectroscopy; a survey of the theory and applications
is given in the book \cite{IL}. In these applications, both nuclear
and spectroscopic, the relevant Hamiltonian is a Lie algebraic
operator in the sense described previously. Finally, the analysis of
a new class of Schr\"odinger operators, the \emph{quasi-exactly solvable} class, was
initiated in late $1980$'s by Shifman, Turbiner and Ushveridze, see
\cite{ST, T,U}. A survey of the theory and
applications of quasi-exactly solvable systems in physics is given in \cite{U2}.

There exists a complete classif\/ication of quasi-exactly solvable Schr\"odinger operators
in one dimension. In two dimensions, this classif\/ication is in
principle complete. Indeed, all the Lie algebraic linear
dif\/ferential operators for which the formal spectral problem is
solvable are known.  The question of determining if the operator is
equivalent to a Schr\"odinger operator
 will be discussed below. The main contribution of our paper is to
 extend these results to three dimensions by giving a partial
classif\/ication of the quasi-exactly solvable Lie algebras of f\/irst order dif\/ferential
operator in three variables, and showing how this can be applied to
the construction of new quasi-exactly solvable Schr\"odinger operators in three
dimensions. Our work is based on the classif\/ication of f\/inite
dimensional Lie algebras of vector f\/ields in three dimensions begun
by Lie in \cite{Lie} and completed by Amaldi in \cite{Ama}.

Let $g^{(ij)}$ be the contravariant metric of the manifold
$\mathbf{M}$ in a local coordinate chart, $\widetilde{g}$ its
determinant and $g$ the determinant of the covariant metric. In that
setting, a Schr\"odinger operator reads locally
as\[\hh_0=-\frac{1}{2}\sum_{i,j=1}^n
\left[g^{ij}\partial_{ij}+\partial_i(g^{ij})\partial_j
-\frac{g^{ij}\partial_i(\widetilde{g})}{2\widetilde{g}}\partial_j\right]+U.
\]
Recall that a quasi-exactly solvable second order operator is not, in general, a
Schr\"odinger operator. However this operator might be equivalent to a
Schr\"odinger operator in a way that preserves the formal spectral
properties of the operators under consideration. The appropriate
notion of equivalence, which will be used throughout our work, is
the following. Two dif\/ferential operators are \emph{locally
equivalent} if there is a gauge transformation $\hh\rightarrow~\mu
\hh \mu^{-1}$, with gauge factor $\mu=e^\lambda$, and a change of
variables relating one to the other. In principle, it is possible to
verify if a general second order dif\/ferential operator $\hh$ is
equivalent to a Schr\"odinger operator with respect to this notion of
equivalence. Indeed, every
 second order linear dif\/ferential operator can be given locally by
\[ \hh=-\frac{1}{2}\sum_{i,j=1}^n g^{ij}\partial_{ij}+\sum_i^n
h^i\partial_i+U.\] If the contravariant tensor $g^{(ij)}$ is
non-degenerate, that is if $g$ does not vanish, the operator can be
expressed as\begin{equation}\label{lap+vec+pot}
\hh=-\frac{1}{2}\Delta+\vec{V}+U,\end{equation} where $\vec{V}=b^i
\partial_i$ is a vector f\/ield. For this operator to be locally equivalent to a
Schr\"odinger operator, the vector f\/ield $\vec{V}$ has to be a
gradient vector f\/ield with respect to the metric $g^{(ij)}$.
Locally, this will be the case if and only if
$\omega=g_{ij}b^jdx^i$, the one form associated to $\vec{V}$, is
closed. For this reason, this condition is named the \emph{closure
condition}. Note that if $\vec{V}=\grad(\lambda)$, the gauge factor
is given by
$e^\frac{\lambda}{2}$.

Given an operator of the form (\ref{liealge}), the closure
conditions can be easily verif\/ied provided the contravariant metric
$g^{(ij)}$ is non-degenerate. Indeed, these conditions can be
written as algebraic constraints on the coef\/f\/icients $C_{ab}$ and
$C_c$, and are the Frobenius compatibility conditions for an
overdetermined system that
will be described later.

An important point to keep in mind is that the class of
quasi-exactly solvable operators is invariant under local
equivalence. Indeed, suppose $\hh$ is a quasi-exactly solvable
operator which is gauge equivalent to an other operator $\hh_0$
under the rescaling $\mu$. If $\hh$ lies in the universal enveloping
algebra of $\g$, whose $\g$-module is $\nn$, one can easily show
that $\hh_0$ is quasi-exactly solvable with respect to the f\/inite
dimensional Lie algebra
\[ \widetilde{\g}=\mu\cdot \g \cdot\mu^{-1}=\big\{  \mu \cdot T \cdot \mu^{-1}\ | \ T\in \g  \big\}
\]
which is isomorphic to $\g$ and posses the f\/inite-dimensional
$\widetilde{\g}$-module
\[\widetilde{\nn}=\mu\cdot \nn=\{  \mu \cdot  h \ | \ h\in \nn   \}.
\]
Note however that the gauge factor is not necessarily unitary. Thus,
a gauge transformation does not necessarily preserve the
normalizability property of the functions in $\nn$. Therefore, the
class of normalizable quasi-exactly solvable operators is not
invariant under our local equivalence. We now give an example of a
normalizable quasi-exactly solvable Schr\"odinger
operator in three variables.

\begin{example}\label{ex}
In this example we
 consider the quasi-exactly solvable Lie
algebra $\g\cong\mathfrak{sl(2)}\times \mathfrak{sl(2)}\times
\mathfrak{sl(2)}$. With the standard notation
$p=\frac{\partial}{\partial x}$, $q=\frac{\partial}{\partial y}$ and
$r=\frac{\partial}{\partial z}$, this Lie algebra representation can
be spanned by the following f\/irst order dif\/ferential operators
\begin{gather*}
     T^1=p, \qquad  T^2=xp,  \qquad  T^3=x^2p-x,  \qquad  T^4=q,  \qquad  T^5=yq,   \\
 T^6=y^2q-y,  \qquad  T^7=r,   \qquad  T^8=zr,  \qquad
     T^9=z^2r-z.
   \end{gather*}
Then, the f\/inite dimensional module of smooth functions
\[ \nn_{m_xm_ym_z}:=\{x^iy^jz^k \,| \ 0\leq i \leq m_x,\  0\leq j  \leq m_y ,\  0\leq k \leq m_z  \}\]
 is a $\g$-module provided $m_x=m_y=m_z=1$. With the following choice of
coef\/f\/icients, one constructs the quasi-exactly solvable operator
\begin{gather*} -2\hh
=(T^1)^2+(T^2)^2+2[(T^3)^2+(T^4)^2+2(T^5)^2+2(T^6)^2]+(T^7)^2+(T^8)^2\\
\phantom{-2\hh=}{} +2(T^9)^2+\{T^1,T^3  \}+\{T^7,T^9\} -2T^4-2T^5-4T^6-8,
\end{gather*}
where $\{T^a,T^b  \}=T^a(T^b)+T^b(T^a)$. The induced contravariant
metric associated to this operator is computed to be the following
positive def\/inite matrix
\begin{equation}\label{metric}g^{(ij)}=\left(
  \begin{array}{ccc}
    (x^2+1)^2 & 0 & (x^2+1)(z^2+1) \\
    0 & y^4+4 y^2 +1 & 0 \\
    (x^2+1)(z^2+1) & 0 & 2(z^2+1)^2 \\
  \end{array}
\right),\end{equation}
 whose determinant is
$g=(x^2+1)^2(y^4+4y^2+1)(z^2+1)^2$. Then, with respect to this
non-degenerate metric, the operator $\hh$ can also be described as
\[-2 \hh=\Delta+\vec{V}+U,\]where $\vec{V}=-2 (x^3+ x+z+ z x^2)p-2(2 y+ y^3)q-2(2z^3+2z+ x+ x z^2)r.$
It is not hard to verify, always with respect to the metric
(\ref{metric}), that  the f\/irst order term $\vec{V}$ is the gradient
of the function $\lambda=-\ln(x^2+1)-1/2\ln(y^4+4y^2+1)-\ln(z^2+1)$.
Hence, by considering the gauge factor
\[\mu=e^{\frac{\lambda}{2}}=(x^2+1)^{-1/2}(y^4+4y^2+1)^{-1/4}
(z^2+1)^{-1/2},\] the operator $\hh$ is gauge equivalent to a
Schr\"odinger operator
\[-2 \hh_0=\Delta+U,\] were the potential is
a rational function of $y$.
Furthermore, it is not hard to show that, after the gauge
transformation, the functions in $\widetilde{\nn}_{111}=\{\ \mu
\cdot x^iy^jz^k\ | \ 0\leq i,j,k \leq 1
 \}$ are square integrable with respect to $\sqrt{g}dxdydz$. Recall here that
 $g$ is the determinant of the covariant metric, hence
 $\sqrt{g}=\mu^2$. Thus, for $i$, $j$, $k$ either $0$ or $1$, one
 can use Fubini's theorem to decompose the integral
\begin{gather*} \iiint_{\mathbb{R}^3}(\mu x^{i}y^{j}z^{k})^2\mu^2
dxdydz =\iiint_{\mathbb{R}^3}\frac{x^{2i}y^{2j}z^{2k}}
{(x^2+1)^{2}(y^4+4y^2+1) (z^2+1)^{2}} dxdydz,
\end{gather*} into
the product of three f\/inite integrals in one variable. Consequently
the operator $\hh_0$ is a~normalizable quasi-exactly solvable Schr\"odinger operator and
it is possible to compute eight eigenfunctions by diagonalizing the
matrix obtained by restricting $\hh$ to $\nn$. For this operator,
one gets two eigenvalues, $-3$ and $1$, both of multiplicity four.
The eight eigenfunctions associated to these two eigenvalues are
respectively,
\begin{gather*}
 \psi_{-3,1}= -1 + x z, \qquad  \psi_{-3,2}=y - x y z,\qquad \psi_{-3,3}=x y + y z,\qquad \psi_{-3,4}=x + z,\\
 \psi_{1,1}= y + x y z,\qquad  \psi_{1,2}= -x + z,\qquad  \psi_{1,3}=-x y + y z,\qquad \psi_{1,4}=1 + x z.
\end{gather*}
Finally, one gets eight eigenfunctions of the Schr\"odinger operator
$\hh_0$
by scaling each of these functions by the gauge factor $\mu$.
\end{example}

In general, there is no a-priori method for testing whether a given
dif\/ferential operator is Lie algebraic or quasi-exactly solvable.
However, one can try to perform a classif\/ication of these operators
under local equivalence using the four-step general method of
classif\/ication
described by Gonz\'alez-L\'opez, Kamran and Olver in \cite{const}.

The f\/irst step toward the classif\/ication of normalizable quasi-exactly solvable
Schr\"odinger opera\-tors is to classify the f\/inite dimensional Lie
algebras of f\/irst order dif\/ferential operators up to dif\/feomorphism
and rescaling. Then, the task is to determine which of these
equivalence classes admit a f\/inite dimensional $\g$-module $\nn$ of
smooth functions.
 Then, from the quasi-exactly solvable Lie algebras
found in the second step, one can construct second order
dif\/ferential operator as described in (\ref{liealge}) from any
choice of coef\/f\/icients $C_{ab}$, $C_c$, and $C_0$. The third step
consists to determine which of these operators are equivalent to
Schr\"odinger operators and this can be performed by verifying the
closure condition. Finally, the last step in this classif\/ication
problem is to check if  the functions contained in the
$\widetilde{\g}$-module $\widetilde{\nn}$
 are square integrable.

As mentioned previously,  the entire classif\/ication has been
established in one dimension. In the scope of the f\/irst two steps,
every quasi-exactly solvable Lie algebra is locally equivalent to a
subalgebra of the Lie algebra
\[\g_n={\rm Span}\left\{ \ \frac{\partial}{\partial x},x\frac{\partial}{\partial
x},x^2\frac{\partial}{\partial x}-nz \ ,1\right\},
  \] where $n$ is a non negative integer, see \cite{QES} for more details.
   Then, once a second order dif\/ferential operator is constructed,
  since all one forms are closed in one dimension, such operator will always be equivalent to
a Schr\"odinger operator, reducing the third step to a trivial step.
Finally Gonz\'alez-L\'opez, Kamran and Olver  determined in \cite{oneD}
necessary and suf\/f\/icient conditions for the normalizability of the
eigenfunctions of the quasi-exacly solvable
Schr\"odinger operators.

In two dimensions, the f\/irst two steps of the classif\/ication problem
were determined by the same authors in \cite{QES2c} and \cite{La2r}.
Based upon Lie's classif\/ication of Lie algebras of vector f\/ields,
see~\cite{Lie}, a complete classif\/ication of the quasi-exactly
solvable Lie algebras $\g$ of f\/irst order dif\/ferential operators,
together with their f\/inite dimensional $\g$-modules, was completed.
The case of two complex variables is discussed in the f\/irst two
papers while the third paper completed the classif\/ication by
considering operators on two real variables. However,
 the last two steps are not yet completed but a wide variety of normalizable
quasi-exactly solvable Schr\"odinger operators has been exhibited, see for instance
\cite{const, QES} and~\cite{La2r}.

In the next section, a partial classif\/ication of quasi-exactly solvable Lie algebras
of f\/irst order dif\/ferential operators in three dimensions is given.
While these two f\/irst steps were successfully completed in one and
two dimensions, only part of this work is now done in three
dimensions. However, these new quasi-exactly solvable Lie algebras can be used to seek
new quasi-exactly solvable Schr\"odinger operators in three dimensional space. The last
section of this paper is devoted to the description of new quasi-exactly solvable
Schr\"odinger operators in three dimensions. Eigenvalues are also
computed for two families of Schr\"odinger operators. These
eigenvalues are part of the spectrum of the operators and their
eigenfunctions, together with their nodal surfaces, are exhibited.
In addition, a connection is made between the separability theorem
proved in \cite{Moi} and the quasi-exactly solvable Schr\"odinger operators on f\/lat
manifold. The quasi-exactly solvable models obtained in our paper are new as far as we
can tell. In particular they are not part of the list of
multi-dimensional quasi-exactly solvable models obtained in \cite{U2} by the method of
inverse separation of variables.

\section[Classification of quasi-exactly solvable Lie algebras of first order differential operators]{Classif\/ication of quasi-exactly solvable Lie algebras\\ of f\/irst order dif\/ferential operators}

\subsection[Lie algebras of first order differential operators]{Lie algebras of f\/irst order dif\/ferential operators}

Our goal in this section is to give a partial classif\/ication of quasi-exactly solvable
Lie algebras of f\/irst order dif\/ferential operators in three
dimensions. A f\/irst step toward this goal is to obtain a~classif\/ication of the f\/inite dimensional Lie algebras $\g$ of f\/irst
order dif\/ferential operators. After this is done, the next step is
to impose the existence of an explicit f\/inite dimensional
$\g$-module~$\nn$ of smooth functions. To this end, we will f\/irst
summarize the basic theory underlying
the classif\/ication of Lie algebras of f\/irst order dif\/ferential operators.

 For $\textbf{M}$ an
$n$-dimensional manifold, we denote by $\fm$ the space of smooth
real-valued functions and $\vm$ the Lie algebra of vector f\/ields on
$\m$. The space $\fm$ form a $\vm$-module under the usual derivation
$\eta \rightarrow v(\eta)$, where $ v$ is a vector f\/ield in $\vm$
and $\eta $ a function in $\fm$. The Lie algebra of f\/irst order
dif\/ferential operators $\dm$ can be described as a~semidirect
product of these two spaces, $\dm=\vm\ltimes \fm$. Indeed, each
element $T$ in~$\dm$ can be written into a sum $T=v+\eta$ and the
Lie bracket is given by
 \begin{equation}\label{cocycle} [T^1,T^2]=[v^1,v^2]+v^1(\eta^2)-v^2(\eta^1), \qquad \textrm{where}\quad
 T^i=v^i+\eta^i \in \dm.
\end{equation} Note that the space $\fm$ is also a $\dm$-module with $T(\zeta)=v(\zeta)+\eta\cdot\zeta.$
Consequently, any f\/inite dimensional Lie algebra of f\/irst order
dif\/ferential operators $\g$ can be written as
 \begin{equation}\label{lie-alge2} T^1=v^1+\eta^1,\qquad \dots,\qquad T^s=v^s+\eta^s,\qquad
 T^{s+1}=\zeta^{1},\qquad \dots,\qquad T^{s+r}=\zeta^{r},
\end{equation}
where $v^1,\dots,v^s$ are linearly independent vector f\/ields spanning
$\h\subset\vm$, a $s$-dimensional Lie algebra and where the
functions $\zeta^1,\dots,\zeta^r$ act as multiplication operators and
span $\mm\subset \fm$ a f\/inite dimensional $\h$-module. Note that
restrictions need to be imposed to the functions $\eta^i$ for $\g$
to be a Lie algebra. Indeed, without the cohomological conditions
that will be described below, the Lie bracket
given in (\ref{cocycle}) does not necessarily return an element in the Lie algebra $\g$.

For $T=v+\eta$, we def\/ine a $1$-cochain  $F:\h\rightarrow \fm$ by
the linear map $\langle F;v \rangle=\eta$. Since any function
$\zeta\in \mm$ can be added to $T$  without changing the Lie algebra
$\g$, this map is not well def\/ined.  To deal with this issue, we
should therefore interpret $F$ as a $\fm/\mm$-valued $1$-cochain.
Thus, from the Lie bracket given in (\ref{cocycle}), it is
straightforward to see that $\g$ is a Lie algebra if and only if the
$1$-cochain $F$ satisf\/ies the bilinear identity
\begin{equation}\label{commutator}
v^i\langle F;v^j \rangle-v^j\langle F;v^i \rangle-\langle
F;[v^i,v^j] \rangle \in \mm,  \qquad v^i,v^j\in \h.
\end{equation}
In terms of Lie algebra cohomology, this condition can be restated
as follow, $ \langle \delta_1 F; v^i,v^j \rangle \in \mm$ for all
$v^i,\ v^j$ in $\h$, i.e.\ $F$ is a $\fm/\mm$-valued
$1$-cocycle on $\h$. (See \cite{fulton} for a detailed
description of Lie algebra cohomology.)

 This classif\/ication of Lie
algebras of f\/irst order dif\/ferential operators would not be complete
without considering the local equivalences between the Lie algebras.
Indeed, if a gauge transformation with gauge factor $\mu=e^\lambda$,
is performed  on an operator $T=v+\eta$ in $\g$, the resulting
dif\/ferential operator $\widetilde{T}=e^\lambda \cdot T \cdot
e^{-\lambda}=v+\eta-v(\lambda) $
 will only dif\/fer from $T$ by the addition of a
multiplication operator $v(\lambda)$. Again, this can be expressed
in cohomological terms. Indeed, under the $0$-coboundary map
$\delta_0 :\h\rightarrow\fm/\mm$ def\/ined by $\langle \delta_0
\lambda;v \rangle=v(\lambda)$, the multiplication factor
$v(\lambda)$ can be interpreted as the image, or the $0$-coboundary,
of the function $\lambda$.  Hence, combining these two observations,
it is possible to conclude that the map $F$ is an element in
$H^1(\h,\fm/\mm)=\ker \delta_1/ \textrm{Im} \delta_0$. Thus, if two
dif\/ferential operators $\g$ and $\widetilde{\g}$ are equivalent with
respect to a change of variables $\varphi$ and a gauge
transformation given by $\mu=e^\lambda$, these two operators will
correspond to equivalent triples $(\h,\mm,[F]),$ and
$(\widetilde{\h},\widetilde{\mm},[\widetilde{F}])$, where
$\widetilde{\h}=\varphi_*(\h)$, $\widetilde{\mm}=\varphi_*(\mm)$,
and $ \widetilde{F}=\varphi_* \circ F \circ
\varphi_*^{-1}+\delta_0\lambda$.
This is summarized in the following theorem.

\begin{theorem}
There is a one to one correspondence between equivalence classes of
finite dimensional Lie algebras $\g$ of first order differential
operators on $\m$ and equivalence classes of triples $(\h,\mm,[F]),$
where\begin{enumerate}\itemsep=0pt
       \item[1)] $\h$ is a finite dimensional Lie algebra of
       vector fields;
       \item[2)] $\mm$ is a finite dimensional $\h$-module of
       functions;
       \item[3)] $[F]$ is a cohomology class in $H^1(\h,\fm/\mm)$.
     \end{enumerate}
\end{theorem}

Hence the general classif\/ication of f\/inite dimensional Lie algebras
 of f\/irst order dif\/ferential operators $\g$ can be bring down to the classif\/ication of triples $(\h,\mm,[F])$
under local changes of variables.

In three dimensions, a complete local classif\/ication of the f\/inite
dimensional Lie algebras of vector f\/ields $\h$ has been established
by Lie in \cite{Lie} and Amaldi in \cite{Ama}. Lie's classif\/ication
distinguishes between the imprimitive Lie algebras, for which their
exists an invariant foliation of the manifold, and the primitive Lie
algebras, for which no such foliation exists. Lie's work gives a
description of the eight dif\/ferent classes of primitive Lie algebras
and, based under the possible foliations of the manifold, the
imprimitive Lie algebras are subdivided
 into the following three types:
\begin{enumerate}\itemsep=0pt
  \item[I.] The manifold admits locally an invariant foliation by surfaces that does not decompose into
  a foliation by curves.
  \item[II.] The manifold admits locally an invariant foliation by curves not contained in a foliation by surfaces.
  \item[III.] The manifold admits locally an invariant foliation by surfaces that does decompose into
  a~foliation by curves.
\end{enumerate}
Observe that these three types are not necessarily exclusive. For
instance, the Lie algebra $\h=\{ p, q, xq, xp-yq, yp, r \}$
belongs to the f\/irst two types. The underlying manifold
$\mathbb{R}^3$ admits a f\/irst indecomposable foliation by planes
$\Delta:=\{ z=\textrm{constant} \}$ and also admits an second
invariant foliation by straight lines $\Phi:=\{ x=\textrm{constant}
\} \cap \{ y=\textrm{constant} \} $ not contained in any invariant
surfaces. Lie classif\/ied the algebras of type~I and~II, giving
respectively twelve and twenty-one  dif\/ferent classes of Lie
algebras. Few years latter, the $103$ classes of Lie algebras of the
third type were exhibited by Amaldi.

The number of f\/inite dimensional Lie algebras
 of vector f\/ields $\h$ is large and it did not seem reasonable to
 consider all the $154$ classes.  For this f\/irst classif\/ication attempt,
 we have chosen to
focus on the algebras which seem promising in our aim to
 construct new quasi-exactly solvable  Schr\"odinger operators. The selection was made upon
 the following criteria.

We f\/irst narrowed our choice based on the results given in
\cite{Moi}; provided the Lie algebra $\g$ is imprimitive and its
invariant foliation consists of surfaces, one can show, adding some
other hypothesis on the metric induced, that a Lie algebraic
Schr\"odinger operator generated by $\g$ separates partially in either
Cartesian, cylindrical or spherical coordinates. Since such algebras
are good candidates for generating interesting quasi-exactly solvable Schr\"odinger
operators, we restricted our search on the type~I and  type~III imprimitive algebras. In this paper, the classif\/ication of the
twelve type~I Lie algebras is entirely performed while, for the
type~III Lie algebras, we focused on some of the most general Lie
algebras. Since the induced metric $g^{(ij)}$ needs to be
non-degenerate, the type~III Lie algebras involving only one or
two of the three partial derivatives were discarded. Finally we
selected our algebras among those that contain other type~III
algebras
 as subalgebras.

\subsection[Classification of Lie algebras of first order differential operators]{Classif\/ication of Lie algebras of f\/irst order dif\/ferential operators}

Using the equivalence between the Lie algebras of f\/irst order
dif\/ferential operators $\g$ and triples $(\h,\mm,[F])$, it is
possible to determine the Lie algebras $\g$ from the selected Lie
algebras of vector f\/ields $\h$. But f\/irst, recall that the second
step in the classif\/ication of quasi-exactly Schr\"odinger operators is
to determine which of these Lie algebras of f\/irst order dif\/ferential
operators $\g$ are quasi-exactly solvable. It is not hard to see that if $\g$ is
quasi-exactly solvable with non trivial f\/ixed module $\nn$, the Lie algebra $\g$ is
f\/inite dimensional if and only if $\mm$ is the module of constant
functions, see \cite{QES2c} for details. Therefore, instead of
working on the general classif\/ication of Lie algebras of f\/irst order
dif\/ferential operators $\g$, we will restrict our work to the
equivalence classes of triples $(\h,\{1\},[F]).$  Thus, for each of
the selected Lie algebras $\h$, we f\/irst seek for the possible
cohomology classes, $[F]$ in $H^1(\h,\fm/\{1\})$. Once this is done,
in the scope of the second step, it will be left to f\/ind if there
exists an explicit f\/inite dimensional $\g$-module $\nn$, where $\g$
is the Lie algebra equivalent to the triple $(\h,\{1\},[F]).$ Note
that, as in lower dimensions, the existence of a nontrivial module
$\nn$ will impose a ``quantization'' condition on $[F]$. Indeed, for
each of the Lie algebras worked out in this paper, the possible
values for the functions in $[F]$ can only be taken in a discrete
set. For detailed results related to the quantization
 of cohomology, see \cite{Quant} and \cite{Quant2}.

\subsubsection[Classification of the  cohomology classes ...]{Classif\/ication of the  cohomology classes  $\boldsymbol{[F]}$ in $\boldsymbol{H^1(\h,\fm/\{1\})}$}

To determine the possible cohomology classes, we f\/irst start with
$[F]$ as general as possible. For every $v$ in the Lie algebra $\h$,
we denote the value of the $1$-cocycle $\langle F;v \rangle$ by
$\eta_v$, and $\eta_v$ can be any function in $\fm/\{ 1\}$. Our aim
is to f\/ind the most general $1$-cocycle $F$, that is the most
general functions $\eta_v$, satisfying the restrictions imposed by
the $1$-cocycle conditions $(\ref{commutator})$. Then, using the
$0$-coboundary map, we try to describe the class $[F]$ with
representatives $\eta_v$ as simple as possible. Finally, if
$\{v^1,\dots,v^r \}$ is a basis for $\h$, the set $\{
v^1+\eta_{v^1},\dots,v^r+\eta_{v^r}\}$ will be a basis for the Lie
algebra $\g$. Note that in this process, one can alternate the use
of the $1$-cocycle restrictions with the use of the $0$-coboundary
cancellations. For instance, if the element $p$ belongs to the
algebra~$\h$, the function $\langle F;p \rangle=\eta_p$ can be
annihilated by the image of the function $\Psi_p=\int \eta_p dx$
under the $0$-coboundary map. Indeed $\langle \delta_0 \Psi_p;p
\rangle=p(\int \eta_p dx)=\eta_p$ and $\widetilde{F}=F-\delta_0
\Psi_p$ belong to $[F]$. Thus, we can assume the function $\eta_p$
to be equivalent to the zero function. Then for another vector f\/ield
$v$ in $\h$, using the $1$-cocycle restriction for the pair $(p,v)$,
that is
\[p\langle F;v \rangle-v\langle F;p \rangle-\langle F;[p,v] \rangle
=p\eta_v-0-\eta_{[p,v]}\in \{1\}, \] one obtains conditions on the
two functions $\eta_v$ and $\eta_{[p,v]}$. Once again, one might try
to absorb part of the function $\eta_v$ with $\delta_0 \Psi_v$, the
 image of another function $ \Psi_v$. Note that,
in order to maintain $\eta_p\equiv 0$, a restriction is imposed on
$\Psi_v$. Indeed, when the $1$-cocycle $F+\delta_0 \Psi_v$ is
applied to $p$, we have to avoid reintroducing a function for
$\eta_p$. Thus we need to consider only the functions $\Psi_v$ for
which
 $\langle \delta_0p;\Psi_v \rangle = (\Psi_v)_x$ is a constant
function. Then, to complete the determination of $[F]$, the same
process is preformed to every vector f\/ield of $\h$, with some care
in the choices of the $0$-coboundary maps, avoiding to undo the
simplif\/ications
 done in the previous steps.

The results of this partial classif\/ication of cohomology classes
$[F]$, that gives a partial classif\/ication of Lie algebras  of
dif\/ferential operators $\g$, are summarized in Tables $1$ and $2$ at
the end of this section. The f\/irst table gives a $1$-cocycle
representative for the twelve type~I Lie algebras and Table~2
exhibits the results for some general Lie algebras of vector f\/ields
among the type~III Lie algebras. For these two tables, the
classif\/ication numbers, given respectively by Lie and Amaldi, sit in
the the f\/irst column. The second column gives a basis for the Lie
algebra~$\h$ and the third column exhibits the f\/irst order
dif\/ferential operators $v+\eta_v$ for which $\langle F;v
\rangle=\eta_v$ in not trivial. $\eta_v$ is taken to be the simplest
representative and when the function $\eta_v$ is trivial, the
dif\/ferential operator is simply the vector f\/ield exhibited in the
second column.

It would be impractical to present the details of the computations
in all cases. Furthermore, the arguments are quite similar for all
 Lie algebra $\h$ of vector f\/ields. So, for brevity's sake, we
will only give the details for two of the selected Lie algebras. The
chosen examples illustrate well the general process and will give to
the reader a good
idea of how the calculations proceed in general.

{\bf Type I, case 1.}
 This Lie algebra $\h$ is spanned by the eight vector f\/ields
$p$, $q$, $xp$, $yq$, $xq$, $yp$, $x^2p+xyq$ and $ xyp+y^2q$. The vector f\/ield $p$
belongs to the Lie algebra, hence, as mentioned previously, the
function $\eta_p$ can be assumed to be zero. The $1$-cocycle
condition for the pair $(p,q)$ imposes the following restriction \[
\langle \delta_1 F; p,q \rangle=(\eta_q)_x-(\eta_p)_y-\langle F;
[p,q]\rangle=(\eta_q)_x \in \{ 1\}.\] Thus $\eta_q=c_qx+h_q(y,z)$,
where $c_q$ is a constant.
 Hopefully, the function $h_q(y,z)$ can be absorbed by the image, under
the $0$-coboundary map, of the function $\Psi_q=\int h_q(y,z) dy$.
Since $(\Psi_q)_x$ is zero, the $0$-coboundary of $\Psi_q$ will not
af\/fect  $\eta_p$.

 Similarly, by considering the pair $(p,xp)$, one
concludes that $\eta_{xp}=c_{xp}x+h_{xp}(y,z)$, where $c_{xp}x$ can
be canceled, without changing the previous functions, by the
$0$-coboundary of the function $\Psi_{xp}=c_{xp}x$. Then, for the
pair $(q,xp)$, the restriction reads as \[\langle \delta_1 F; q,xp
\rangle=(\eta_{xp})_y-x(\eta_q)_x-\langle F;
[q,xp]\rangle=(h_{xp}(y,z))_y-x\cdot c_q \in \{ 1\}.\] Necessarily,
since $h_{xp}$ depends only on $y$ and $z$,  the constant $c_q$ has
to be zero and the function $h_{xp}(y,z)$ is forced to be of the
form $d_{xp}y+K(z)$, where $d_{xp}$ is a constant. Thus, at this
point, $\eta_p=0$, $\eta_q=0$
and $\eta_{xp}=d_{xp}y+K(z)$.

Consider now  the three vector f\/ields $yq$, $xq$ and $yp$. If we
pair each of them with $p$ and $q$, from the six $1$-cocycle
restrictions, one obtains directly the following
\begin{gather*}
\eta_{yq}=c_{yq}x+d_{yq}y+k_{yq}(z),\!\!\qquad
\eta_{xq}=c_{xq}x+d_{xq}y+k_{xq}(z),\!\!\qquad
\eta_{yp}=c_{yp}x+d_{yp}y+k_{yp}(z).
\end{gather*} With the image
of the function $\Psi_{yq}=d_{yq}y$, the function $\eta_{yq}$ can be
reduced to $\eta_{yq}=c_{yq}x+k_{yq}(z)$ without undoing the
previous work. From the restriction associated to the pair
$(xp,yp)$, one easily check that
\[ \langle \delta_1 F; xp,yp
\rangle=x(\eta_{yp})_x-y(\eta_{xp})_x+\eta_{yp}=x\cdot
c_{yp}+c_{yp}x+d_{yp}y+k_{yp}(z) \in \{1\},\] forcing $c_{yp}$ and
$d_{yp}$ to be zero and $k_{yp}(z)$ to be a constant function.
Similarly, by considering the pair $(xp,yq)$, one obtains that the
function
$x\cdot c_{yq}-y\cdot d_{xp}$ must be constant, hence  $c_{yq}$ and
$d_{xp}$ are zero. To completely determine the functions $\eta_v$
for these three vector f\/ields, two restrictions, associated to the
pairs $(xp,xq)$ and $(yp,xq)$, must be verif\/ied. The f\/irst imposes
that $x(\eta_{xq})_x-x(\eta_{xp})_y-\eta_{xq}=x\cdot
c_{xq}-c_{xq}x-d_{xq}y-k_{xq}(z)$ must be constant. Thus it leaves
no choice but to take $d_{xq}$ as the constant zero and $k_{xq}(z)$
as a constant function. Finally, the last restriction forces
$y(\eta_{xq})_x-x(\eta_{yp})_y-\eta_{yq}+\eta_{xp}=y\cdot
c_{xq}-k_{yq}(z)+K(z)$ to be a constant, hence~$c_{xq}$ must be zero
while $k_{xq}(z)$ must be equal, modulo the constant functions, to
the function $K(z)$. Putting together these restrictions, the image
of the $1$-cocycle $F$ for the f\/irst six vector f\/ields of~$\h$ can
be described as $\eta_p=\eta_q=\eta_{xq}=\eta_{yp}=0$ and
$\eta_{xp}=\eta_{yq}=K(z)$. One easily checks that the remaining two restrictions are satisf\/ied.

To determine completely the $1$-cocycle $F$, it remains to f\/ind its
images for the two vector f\/ields $T:=x^2p+xyq$ and $Q=xyp+y^2q$. For
$\eta_T$, three restrictions are needed to reach that
$\eta_T=3xK(z)$. Indeed, from the pair $(p,T)$, the cocycle
condition forces the following equality
$(\eta_T)_x-2\eta_{xp}-\eta_{yq}=c_T$, where $c_T$ is a constant. It
is not hard to see that $\eta_T$ must be equal to
$3xK(z)+c_Tx+h_T(y,z)$. From the pair $(q,T)$, we get similarly that
$\eta_T=3xK(z)+c_Tx+d_Ty +k_T(z)$. Finally, the restriction for the
pair $(xp,T)$ leads to\[x(\eta_T)_x-T(\eta_{xp})-\eta_T=x\cdot3
K(z)+x \cdot c_T-(3xK(z)+c_Tx+d_Ty +k_T(z))\in \{1\}.\] Hence the
constant $d_T$ dies out and $k_T(x)$ has to be a constant function.
Note that, with these $3$ restrictions, $\eta_T=3x(K(z)+c_T/3)$,
but, by taking $\eta_{xp}= K(z)+c_T/3$,
 one gets the claimed result. By symmetry on $x$ and $y$, the exact
same arguments lead to $\eta_Q=3yK(z)$. It is then straightforward
to verify that the $1$-cocycle $F$, given by the eight functions
\begin{gather*}
\eta_p=\eta_q=\eta_{xq}=\eta_{yp}=0, \!\!\qquad \eta_{xp}=\eta_{yq}=K(z),\!\! \qquad \eta_P=3xK(z)
\!\!\qquad \textrm{and}\!\! \qquad\eta_Q=3yK(z),
\end{gather*}
 satisf\/ies all the other
$1$-cocycle conditions. Finally, the Lie algebra $\g$ associated to
this triple $(\h,\{1\},[F])$ is the Lie algebra spanned by{\samepage
 \[ \big\{p,q,xp+K(z),yp,xq,yq+K(z),x^2p+xyq+3xK(z),xyp+y^2q+3yK(z),1 \big\},\]
 where $K(z)$ can be any function.}

Fortunately the calculations performed for a given Lie algebra $\h$
can be repeated for any other Lie algebra sharing a subset of
generators with $\h$. Note also that some ad hoc lemma's were used
trough this work to simplify
these calculations. For instance.

\begin{lemma} \label{lemmef} Let $i:\mathbb{R}^2\rightarrow \mathbb{R}^3 $, $(x,y)\mapsto (x,y,z)$
denote the inclusion map and suppose that $\h_0\subset\Gamma(i_*
T\mathbb{R}^2)$, meaning that the generators of $\h_0$ depend on the
variables $x$ and $y$ only. Let~$\h$ be a Lie algebra of vector
fields on $\mathbb{R}^3$ given by $\h=\h_0\oplus\{r,zr,z^2r\}$. If,
for non constant functions $f(x,y)$ and $g(x,y)$, the vector fields
$f(x,y)p$ and $g(x,y)q$ belong to $\h$ and if their associated
images $\eta_{f(x,y)p}$ and $\eta_{g(x,y)q}$ depend on $x$ and $y$
only, then \[H^1(\h,\ff (\mathbb{R}^3 )/\{1\})=H^1(\h_0,\ff
(\mathbb{R}^2)/\{1\})\oplus H^1(\{ r,zr,z^2r\},\ff
(\mathbb{R})/\{1\}).\]
\end{lemma}

\begin{proof} Denote $A:=f(x,y)p$ and $B:=g(x,y)q$.
From the cocycle restrictions associated to the pairs $(A,z^ir)$,
where $i=0,1,2$, we obtain
\begin{gather*}\langle \delta_1 F; A,z^ir
\rangle=A(\eta_{z^ir})-z^i(\eta_A)_z-\langle  F; [A,z^ir] \rangle
=f(x,y)(\eta_{z^ir})_x-0-\langle F,0 \rangle\\
\phantom{\langle \delta_1 F; A,z^ir \rangle}{} = f(x,y)(\eta_{z^ir})_x\in \{ 1 \}.
\end{gather*}
Since $f(x,y)$ is not constant, $(\eta_{z^ir})_x$ must vanish, hence
the functions $\eta_{z^ir}$ depend on $y$ and $z$. In a similar way,
from the restrictions associated to the pairs $(B,z^ir)$ it is
straightforward to conclude that $\eta_{z^ir}=h^i(z)$. Finally, for
any element $v$ in $\h_0$,  the function $\eta_v$ will depend on $x$
and $y$ only. Indeed, since $\eta_{zr}$ depends on $z$ only,
\begin{gather*}\langle \delta_1 F; v,zr
\rangle=v(\eta_{zr})-z(\eta_v)_z-\langle  F; [v,zr] \rangle
=0-z(\eta_v)_z-\langle F,0 \rangle
= -z(\eta_v)_z\in \{ 1 \}.
\end{gather*} Therefore $(\eta_v)_z$ must be zero, forcing the function
$\eta_v$ to depend on $x$ and $y$ only.
\end{proof}

Note that $H^1(\{ r,zr,z^2r\},\ff (\mathbb{R})/\{1\})$ is already
well known. The  $1$-cocycle $F$  associated the Lie algebra $\h=\{
r,zr,z^2r\}$ is determined by three functions and the simplest
representative is given by $\eta_r=0$, $\eta_{z r}=0$ and $\eta_{z^2
r}=dz$, for any constant $d$. Thus, one can use this lemma to
simplify some of the computations required in this classif\/ication
problem. For instance, given~$\h$ the type~I Lie algebra of
vector f\/ields given by case~10 in Table~1, the Lie algebra of
dif\/ferential operators $\g$ built from $\h$  is obtained from a
direct application of this lemma.

{\bf Type I, case 10.}
 The Lie algebra
$\h={\rm Span}\{p,q,xp,yq,xq,yp,x^2p+xyq,xyp+y^2q, r,zr,z^2r  \}$ can be
decomposed as $\h_0\oplus \{r,zr,z^2r  \}$ where $\h_0$ is the case~1 Lie algebra from the same table. It was shown in the previous
calculations that the functions $\eta_{yp}$ and $\eta_{xq}$ are
zero, hence functions on $x$ and $y$ only.  Thus the case~10 Lie
algebra, along with its two vector f\/ields $yp$ and $xq$, satisf\/ies
the requirements of the Lemma~\ref{lemmef}. Therefore, for the
vector f\/ields in the algebra $\h_0$, the values of the $1$-cocycle
depend on $x$ and $y$ only, forcing $K(z)$ to be $c$ a constant
function. It is then obvious that the $1$-cocycle $F$ is def\/ined by
eleven functions, were the three non-zero are given by $\eta_T=cx,$
$\eta_Q=cy $ and $\eta_{z^2r}=dx$, for $c$ and $d$ any constants.
The Lie algebra of f\/irst order dif\/ferential operators $\g$
corresponding to this triple is then
\[  \g={\rm Span}\{p,q,xp,yq,xq,yp,x^2p+xyq+cx,xyp+y^2q+cy, r,zr,z^2r+dz,  1\}.
\]

It should be pointed here that $H^1(\h,\fm/\{1\})$ and $H^1(\h,\fm)$
can also be determined alternatively using isomorphisms given in
\cite{MilsonCoho} and \cite{Miller}. For the case that interests us,
that is $H^1(\h,\fm/\{1\})$, we f\/ix a base point $e$ and denote
$\mathfrak{i}$ the isotropy subalgebra. Provided the existence of a
subalgebra $\mathfrak{a}\subset\mathfrak{h}$ which is complementary
to $\mathfrak{i}$, one can show, see \cite{MilsonCoho}, that
\[ H^1(\h,\fm/\{1\})\cong H^2(\mathfrak{h}/\mathfrak{i}). \]
This isomorphism leads to an explicit method for constructing
$1$-cocycle representatives $F$ in
 $H^1(\h,\fm/\{1\})$. We f\/irst choose
$\alpha$, a $2$-cocycle representative of a class in
$H^2(\mathfrak{h}/\mathfrak{i})$, and, for $\{ v^1,\dots,v^n\}$ a
basis of $\mathfrak{h}$, we denote $\alpha_{ij}=\alpha(v^i,v^j)$. If
$\mathfrak{a}=\{ v^1,\dots,v^m\}$, $\mathfrak{i}=\{
v^{m+1},\dots,v^n\}$, and $c_{ij}^k$ are the structure constants of
the Lie algebra $\mathfrak{h}$, a $1$-cocycle in $H^1(\h,\fm/\{1\})$
will be obtained by solving f\/irst the following $m(m-1)$ equations
\[v^i(f_j)-v^j(f_i)-\sum_k
c_{ij}^k f_k=\alpha_{ij}, \qquad \textrm{for} \quad 1\leq i<j\leq m.\]
Once a non unique solution $f_1,\dots,f_m$ is obtained, the remaining
functions $f_{m+1},\dots, f_n$  are determined as the unique solution
to the $m(n-m)$ equations
\[v^i(f_j)-v^j(f_i)-\sum_k
c_{ij}^k f_k=\alpha_{ij}, \qquad \textrm{for} \quad 1\leq i \leq m,\quad
m+1 \leq  j\leq n,\] with initial conditions
\[f_i(e)=0  \qquad \textrm{for } \quad m+1\leq i \leq n. \]

Note however that this method can not be applied to all the three
dimensional Lie algebras since the existence of the complementary
Lie subalgebra is not guaranteed. For instance, the type III case
17A$_1$ can not be treated using the isomorphism. Indeed, for the
Lie algebra
\[ \mathfrak{h}=\{p,q,xp+zr,yq,x^2p+(2x+az)zr,y^2q \},\]
and the base point $e=(0,0,0)$, the isotropy algebra $\mathfrak{i}$
 is generated by the last four elements and the algebra
 $\mathfrak{a}=\{p,q \}$ fails to be complementary, due to the
 absence of the element $r$ in the Lie algebra. One can easily
 verify that $Z^2(\mathfrak{h}/\mathfrak{i})=\{ \alpha_1\wedge\alpha_3, \alpha_1\wedge\alpha_5,
  \alpha_2\wedge\alpha_4,\alpha_2\wedge\alpha_6\}$ and
  $B^2(\mathfrak{h}/\mathfrak{i})=\{ \alpha_1\wedge\alpha_3,
  \alpha_2\wedge\alpha_4\}.$ One can observe at that point that the
  theorem does not hold, since the dimension of
  $H^2(\mathfrak{h}/\mathfrak{i})$ is two while the dimension of
  $H^1(\h,\fm/\{1\})$ was computed to be three previously. Moreover
  applying the technique to the $2$-cocycle $\alpha=c \cdot \alpha_1\wedge\alpha_5+
  d \cdot\alpha_2\wedge\alpha_6$, one gets a $1$-cocycle that does
  not satisf\/ies all the conditions that were not considered in the
  technique detailed above.

\subsection[Classification of quasi-exactly solvable Lie algebras of first order differential operators and the quantization condition]{Classif\/ication of quasi-exactly solvable Lie algebras\\ of f\/irst order dif\/ferential operators and the quantization condition}

The Lie algebras given in Tables~1 and 2 are the candidates for
being quasi-exactly solvable Lie algebras, i.e.\ we might expect them to admit
$\nn$ a f\/inite dimensional module of smooth functions~$\nn$. In
the investigation for these explicit f\/inite dimensional modules,
some new restrictions are imposed on the $1$-cocycles $F$. Indeed,
as for the quasi-exactly solvable Lie algebras in lower dimensions, it comes out that
a f\/inite dimensional module exists only if the values of the
functions $\eta_v $ are taken in a certain discrete set. For this
reason, this restriction is named quantization condition. The quasi-exactly solvable
Lie algebras and their f\/ixed modules can be found in Tables $3$ and
$4$ for, respectively, the type~I and  the selected type~III
Lie algebras. The f\/irst column use the same classif\/ication numbers
as in
 Tables $1$ and $2$ and a representative for the non-trivial
quantized $1$-cocycles is exhibited in the second column. Finally,
$\nn$, the f\/inite dimensional $\g$-modules of functions are
described in the last column. Once again the detailed calculations
are repetitive and the essence of the work can be
 grasped with one or two examples, together with the following
 general principles.
\begin{enumerate}\itemsep=0pt
  \item A f\/inite dimensional module for the trivial Lie algebra $\g=\{ p \}$
   is def\/ined as an {\it $x$-translation
  module}. For instance, any space spanned by a f\/inite set of functions
  of the form
\[ h=\sum_{i=0}^n g^i(y,z)x^i,\]
along with all their $x$ derivatives, is an $x$-translation
  module. This particular case of $x$-translation
   module is referred as a {\it semi-polynomial $x$-translation
  module}.
The most general $x$-translation module is obtained by a direct sum
  \[
\nn=\bigoplus_{\lambda \in \Lambda} \nn_\lambda, \qquad
\nn_\lambda=\widehat{\nn_\lambda} e^{\lambda  x},  \] where
$\widehat{\nn_\lambda}$ are semi-polynomial $x$-translation modules
and the exponents are taken in a~f\/inite set $\Lambda$, read
\cite{LA2c} for more details. Obviously, the $y$, and the
$z$-translation modules are def\/ined the exact same way.
  \item If the Lie algebra $\g$ under consideration contains the  two dif\/ferential operators $p$ and $ xp$,
  the module $\nn$ will be an $x$-translation module and the  operator $xp$ will impose extra constraints.
Firstly, all the exponents $\lambda$ need be zero. Otherwise, for an
non-zero exponent~$\lambda$, the degree in $x$ of the generating
functions in the
  module $\widehat{\nn_\lambda}$ would be unbounded,
  contradicting the f\/inite dimensionality of $\nn$.
   Moreover, if $h=\sum_{i=0}^n g^i(y,z)x^i$ belongs to
the module $\nn$, the function $xh_x$ also needs to belong to that
module. Note that both functions have the same degree in $x$ and are
linearly independent if $h$ is not a monomial. Thus, by an
appropriate linear combination of these two functions, one can
reduce the number of summands in $h$.  By
  iterating this process, each generating function can be reduced to a~monomial in $x$. Thus, $h=g(y,z)x^i$
  where $g(y,z)$ belongs to $G^{i}$ a f\/inite set of functions in $y$ and
  $z$. Since $\nn$ is a $x$-translation module,
  $h_x=ig(y,z)x^{i-1}$ is also a function in $\nn$\!, hence
  $g(y,z)$ needs to be also contained in $G^{i-1}$.
   Therefore, the module $\nn$
  decomposes into the following direct sum
  \[ \nn=\bigoplus x^ig^i_k(y,z),  \qquad i=0,\dots,n, \qquad k=0,\dots,l_i,
  \] where all the functions $g^i_k(y,z)$ belong to $G^i$ a f\/inite set and where $G^{i}\subseteq G^{i-1}$.

  \item \label{test} Likewise, if  a Lie algebra $\g$ contains the elements $p$, $q$, $xp$, and $yq$,
    a general f\/inite dimensional $\g$-module for this Lie algebra will be at most
  \[ \nn=\bigoplus
  x^iy^jg_k^{i,j}(z), \qquad i=0,\dots,n, \qquad j=0,\dots, m,\qquad k=0,\dots,l_{(i,j)},
  \]  where the functions $g^{i,j}_k(z)$ belong to $G^{(i,j)}$, a f\/inite set of functions of $z$
  satisfying $G^{(i,j)}\subseteq G^{(i-1,j)}\cap
  G^{(i,j-1)}$.

 A simple method to describe these modules is to
  represent each generating function $x^iy^jg^{i,j}(z)$ by a point $(i,j)$, in the Cartesian
 plane. If a
  vertex $(i,j)$ belongs to the diagram, since $\nn$ is an $xy$-translation module,
   the vertices $(i-1,j)$ and
  $(i,j-1)$ must also sit in the diagram. To complete the description,
   a f\/inite set $G^{(i,j)}$ is associated to each of these vertices, with the same
  restriction as above. For instance, such module $\nn$ can be represented by

 \centerline{\includegraphics[width=30mm]{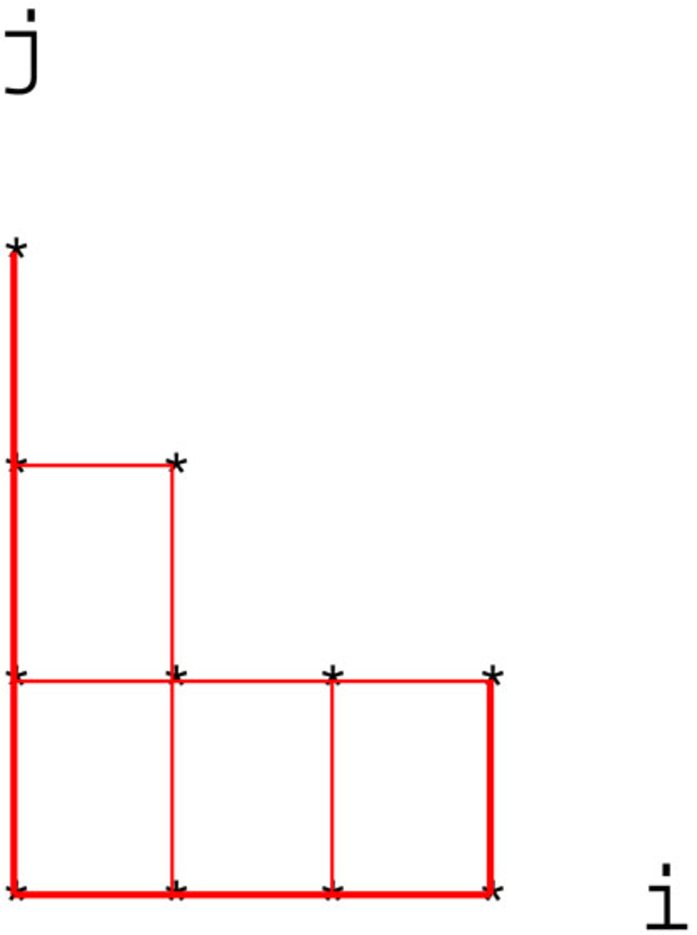}}

 with all the sets $G^{(i,j)}$ being equal to $\{z,e^z\}$, with the exception of $G^{(3,1)}$
 that contains only the function $z$.
 It is then straightforward to verify that this module is indeed
 \begin{gather*} \nn= {\rm Span}\{0,
 z,e^z,xz,xe^z,yz,ye^z,x^2z,x^2e^z,xyz,xye^z,y^2z,\\
\phantom{\nn= {\rm Span}\{}{} y^2e^z,
x^3z,x^3e^z,x^2yz,x^2ye^z,xy^2z,xy^2e^z,y^3z,y^3e^z,x^3yz\},
\end{gather*}and that it is a $\g$-module for the Lie algebra $\g=\{ p,xp,y,yq \}.$

  \item If the Lie algebra $\g$ contains the dif\/ferential operators $p$, $q$, $xp$, $yq$ and $yp$, from the
  three previous principles,
  the generators for a $\g$-module are given by $h=x^iy^jg^{i,j}(z)$. After applying the operator $yp$ on $h$, the resulting function reads as
   $ix^{i-1}y^{j+1}g^{i,j}(z)$. Thus, iterating this operator, we conclude that all the functions
    $x^{i-r}y^{j+r}g^{i,j}(z)$ must belong to~$\nn$, for $r\leq i$.
Since $p$ and $q$ also belong to the algebra, all the functions
$x^ay^bg^{i,j}(y,z)$ with $a\leq i$, $b\leq j$ and $a+b\leq c$ must
belong to
   $\nn$.
 This condition can be expressed by the following inclusion
  $G^{(i,j)}\subseteq G^{(i-1,j)}\cap G^{(i,j-1)}\cap G^{(i-1,j+1)}, $ and observe that
  the f\/irst set
  $G^{(i-1,j)}$ can be omitted without af\/fecting the condition.
   To summarize, the $\g$-module will be at most
   \[ \nn=\bigoplus
  x^iy^jg_k^{i,j}(z), \qquad i=0,\dots,n, \qquad j=0,\dots, m, \qquad k=0,\dots,l_{(i,j)},
  \]  where the functions $g^{i,j}_k(z)$ belong to $G^{(i,j)}$, a f\/inite set of functions  with
  $G^{(i,j)}\!\subseteq\! G^{(i-1,j+1)} \cap
  G^{(i,j-1)}$.

  Once again it is possible to represent such module by a diagram
  along with
  a
  set of functions $G^{(i,j)}$ for each vertex of the diagram.
   The restrictions for
  these sets are $G^{(i,j)}\subseteq G^{(i-1,j+1)} \cap
  G^{(i,j-1)}$ and the
  conditions on the vertices are slightly dif\/ferent from the one in the previous example. Indeed,
   if a
  vertex $(i,j)$, belongs to the diagram, the two vertices
  $(i-1,j+1)$ and  $(i,j-1)$  must also belong to the diagram. Note again that this
   implies that the vertex $(i-1,j)$ also
  lies in the diagram. For instance
  the diagram,

  \centerline{\includegraphics[width=3cm]{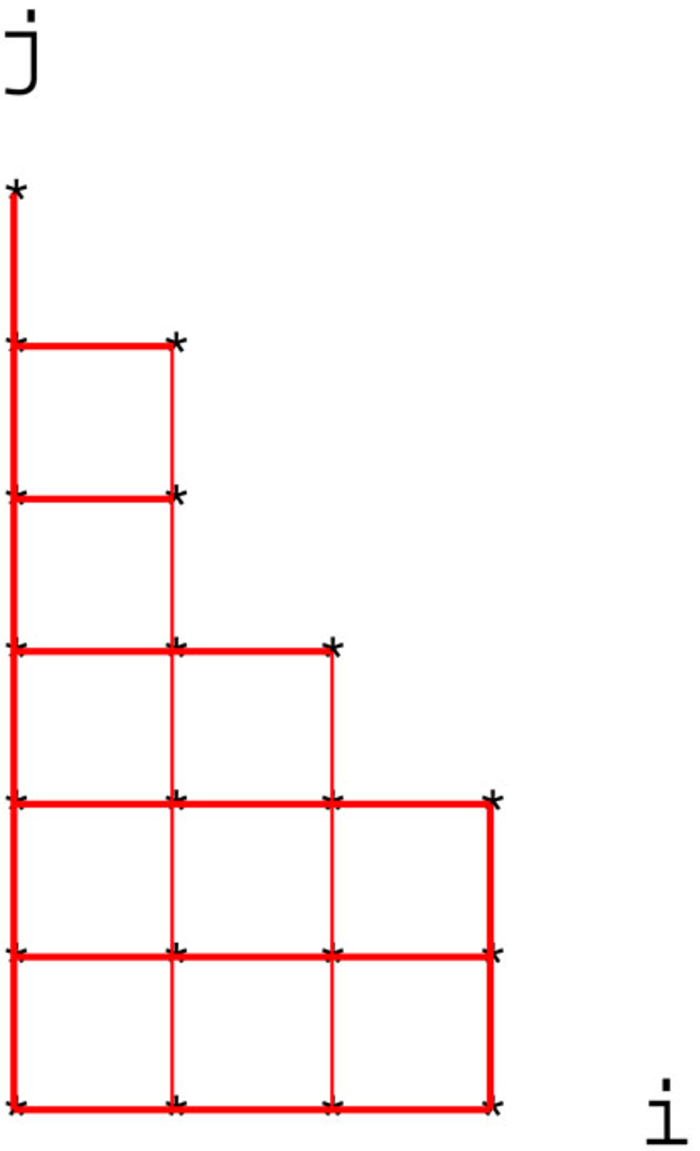}}
together with twenty appropriate sets of functions $G^{(i,j)}$ for
each vertex, would generate
 a $\g$-module for the algebra $\g=\{p,q,xp,yq,yp\}$.

  \item \label{test2} Finally, if the elements $p$, $q$, $xp$, $yq$, $yp$ and $xq$ sit in the Lie algebra under
  consideration, the module will be at most \[ \nn=\bigoplus
  x^iy^jg_k^{i,j}(z), \qquad i+j=0,\dots,n, \qquad k=0,\dots,l_{i,j},
  \]  where the functions $g^{i,j}_k(z)$ belong to $G^{(i+j)}$, a f\/inite set of functions with
  $G^{(l)}\subseteq G^{(l-1)}$. Indeed, consider $h=x^iy^jg^{i,j}(z)$, a generator
   of bi-degree $i+j=c$. Since $xq[h]$ and $yp[h]$ must also lie in $\nn$, all
  the functions $x^ay^bg^{i,j}(z)$ with $a+b=c$ will belong to $\nn$. Hence
  $g^{(i,j)}\in
  G^{(a,b)}$ and, reciprocally,  $g^{(a,b)}\in
  G^{(i,j)}$. Thus for all pairs $(a,b)$ with $a+b=c$,
  the f\/inite sets $G^{(a,b)}$ are identical and it is therefore well def\/ined to pose
  $G^{(a,b)}=G^{(a+b)}.$ Obviously, since $\nn$ is a $xy$-translation module, the following
  inclusions
  hold $G^{(l)}\subseteq G^{(l-1)}$.

For these modules, the possible diagrams are more restricted and
have necessarily the shape of a staircase. Also, instead of
assigning one set of functions to each vertex, such a~set is coupled
to all the vertices having same total degree $i+j$. For instance,
the
 module represented by the diagram

  \centerline{\includegraphics[width=4cm]{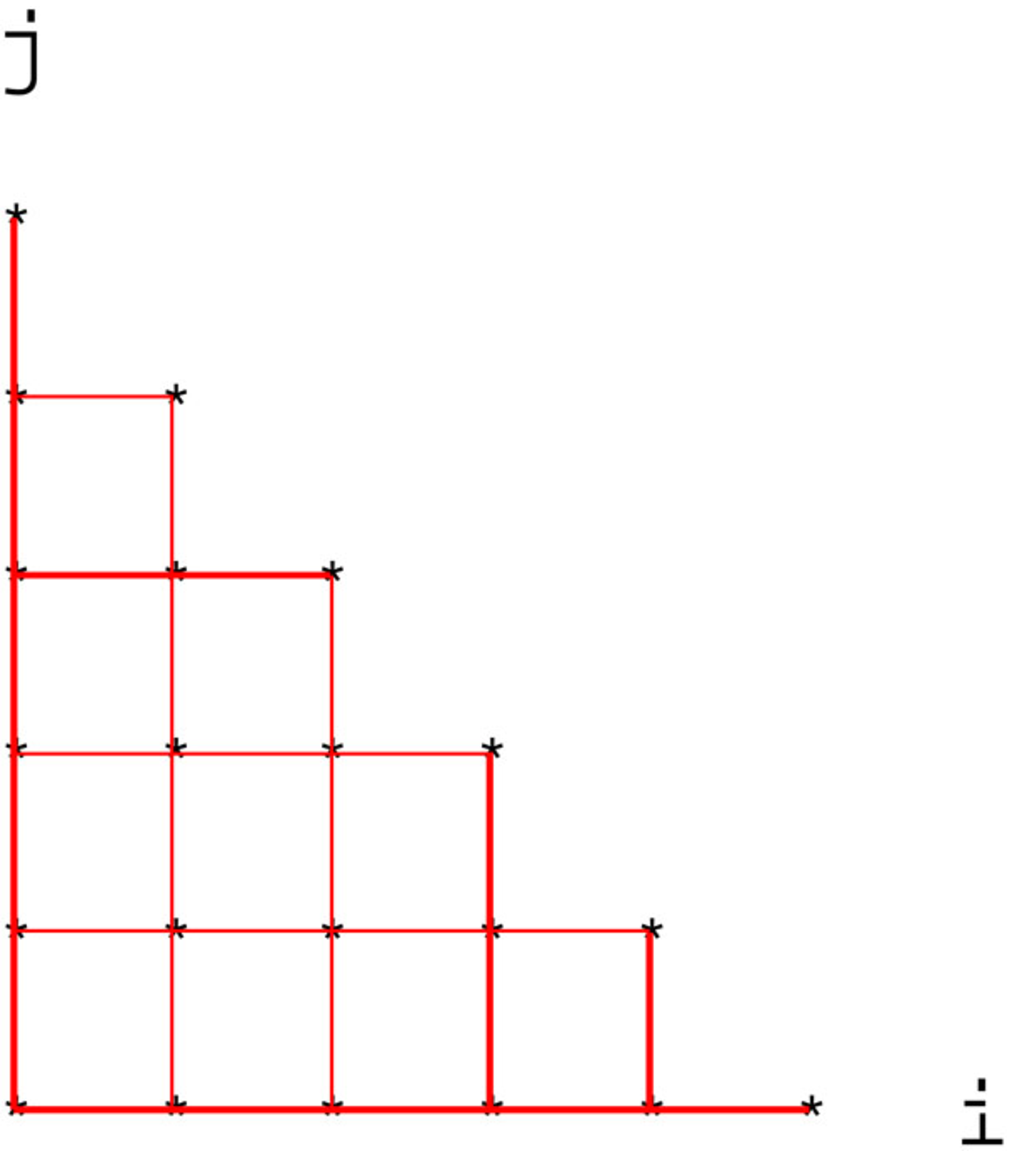}}
will be completely determined after f\/ixing six sets of function in
$z$. Note that this choice must respect the inclusion
$G^{(l)}(z)\subseteq G^{(l-1)}(z)$, for $i=1,\dots,5$.
\end{enumerate}

\begin{table}[hb] \small {\centering
\caption{Cohomology for the type I Lie algebras of  vector
f\/ields, $m=\{1\}$.}

\vspace{2mm}

\begin{tabular}{|l l l|}
  \hline
\tsep{0.5ex}  & Generators  & Cocycles \bsep{0.5ex}\\ \hline \hline
\tsep{0.5ex} 1 &    $\{p,q,xp,yq,xq,yp, x^2p+xyq,xyp+y^2q  \}$ & $xp+K(z), x^2p+xyq+xK(z)$,
\\ & & $yq+K(z), xyp+y^2q+yK(z)$\bsep{0.5ex}\\ \hline
\tsep{0.5ex}2 &    $\{xq, xp-yq,yp,$  &$0$  \\
 & $Z^1(z)p,\dots,Z^l(z)p,Z^1(z)q,\dots,Z^l(z)q   \}$ & \bsep{0.5ex}\\ \hline
\tsep{0.5ex}3 &    $\{xq, xp-yq,yp, xp+yq, $    &$xp+yq+K(z)    $ \\
 & $Z^1(z)p,\dots,Z^l(z)p,Z^1(z)q,\dots,Z^l(z)q$\}&\bsep{0.5ex}\\ \hline
\tsep{0.5ex}4 &    $\{p,q,xp,yq,xq,yp, x^2p+xyq,xyp+y^2q,r  \}$  &$x^2p+xyq+cx$, $xyp+y^2q+cy$  \bsep{0.5ex}\\ \hline
\tsep{0.5ex}5 &    $\{xq, xp-yq,yp, z^ke^{\lambda_lz}p,z^ke^{\lambda_lz}q,r\}$,  &$0$ \\
& $ k\leq  n_l,$ $l=0,\dots,b  $ & \bsep{0.5ex}\\ \hline
\tsep{0.5ex} 6 &    $\{xq, xp-yq,yp, xp+yq, $   &$xp+yq+cz     $ \\
 & $ z^ke^{\lambda_lz}p,z^ke^{\lambda_lz}q,r\}$, $ k\leq  n_l,$ $l=0,\dots,b
 $ &\bsep{0.5ex}\\ \hline
\tsep{0.5ex}7 &    $\{p,q,xp,yq,xq,yp,x^2p+xy, q,xyp+y^2q, r,zr \}$    &$x^2p+xyq+cx $, $xyp+y^2q+cy$ \bsep{0.5ex}\\ \hline
\tsep{0.5ex} 8 &    $\{xq, xp-yq,yp, p,zp,\dots,z^lp,$  &$0$  \\
 & $q,zq,\dots,z^lq, r, zr+a(xp+yq) \}$ & \bsep{0.5ex} \\ \hline
\tsep{0.5ex}9 &    $\{xq, xp-yq,yp, xp+yq, p,zp,\dots,z^lp,$      & $0$ \\
 & $q,zq,\dots,z^lq,r,zr \}$, $k\leq  n_l,$ $l=0,\dots,b  $ & \bsep{0.5ex}\\ \hline
\tsep{0.5ex} 10 &    $ \{p,q,xp,yq,xq,yp,$    & $x^2p+xyq+cx $, $xyp+y^2q+cy$,\\
& $x^2p+xyq,xyp+y^2q, r,zr,z^2r$ \}&  $z^2r+dz$\\ \hline
\tsep{0.5ex} 11 &    $\{xq, xp-yq,yp, p,zp,\dots,z^lp,q,zq,\dots,z^lq, $  & $z^2r+az(xp+yq)+cz$ \\
& $  r, zr+\frac{a}{2}(xp+yq), z^2r+az(xp+yq)\} $&   \bsep{0.5ex}\\
\hline
\tsep{0.5ex} 12 &    $\{xq, xp-yq,yp, xp+yq, p, zp,\dots,z^lp,$      & $z^2r+az(xp+yq)+cz  $  \\
 & $q,zq,\dots,z^lq,r,zr, z^2r+az(xp+yq)  \}$ & \bsep{0.5ex}\\
  \hline \hline
\end{tabular}

}

\vspace{2mm}

Note here that $b$, $l$ and $n_l$ are positive integers, $a$, $c$,
$d$, $k$ and $\lambda_l$ are arbitrary constants and
$Z^1(z),\dots,Z^l(z)$ and $K(z)$ functions of $z$.
\end{table}

\begin{table}[t]\small {\centering

\caption{Cohomology for some type III Lie
algebras of vector f\/ields, $m=\{1\}$.}

\vspace{2mm}

\begin{tabular}{| l l l|}
  \hline
 \tsep{0.5ex} & Generators & Cocycles  \bsep{0.5ex} \\ \hline \hline
\tsep{0.5ex}4A &    $ \{p,yq,q,xq,\dots,x^tq+r,\dots,x^iq+\binom{t}{i}x^{t-i}r,\dots,$  & $0$ \\
 & $x^sq+\binom{t}{s}x^{s-t}r,xp-tzr,yq+zr\}$, $0\leq t \leq s$& \bsep{0.5ex} \\ \hline
\tsep{0.5ex}4C &    $ \{q,xq,\dots,x^sq,p,yq,xp,x^ly^{n-b}r,zr \}$,   & $0$ \\
 & $ 0 \leq b\leq n,$ $l\leq l_0+sb$     & \bsep{0.5ex}\\ \hline
 \tsep{0.5ex}4D &    $ \{q,xq,\dots,x^sq,p,yq,xp,r,zr,z^2r \}$   & $z^2r+cz$ \bsep{0.5ex}\\ \hline
\tsep{0.5ex}5A$^*$ &    $ \{q+r,xq+xr,\dots,x^sq+x^sr,p,$    &$x^2p+sxyq+sr+cx$ \\
 & $xp,yq+zr,x^2p+sxyq+sxzr \}$& \bsep{0.5ex}\\ \hline
\tsep{0.5ex}5C &    $ \{q,xq,\dots,x^sq,p,yq,xp,x^2p+sxyq+(l_0+s\tilde{n})xzr, $ &$x^2p+sxyq+(l_0+s\tilde{n})r+cx$  \\
 &  $x^ly^{\tilde{n}-b}r,zr  \}$,  $ 0 \leq b\leq \tilde{n},$ $l\leq l_0+sb$    & \bsep{0.5ex}\\ \hline
\tsep{0.5ex}5D &    $ \{q,xq,\dots,x^sq,p,yq,xp,x^2+sxyq,r, zr,z^2r\}$    &$x^2p+sxyq+cx$, $z^2r+dz$ \bsep{0.5ex} \\ \hline
\tsep{0.5ex}7C &    $ \{ p,2xp+yq,x^2p+xyq,x^ly^{-n}r,zr \}$, $ 0 \leq l\leq n$   &$x^2p+xyq+cy^2$ \bsep{0.5ex}  \\ \hline
 \tsep{0.5ex}17A$_1$ &    $ \{ p,q,xp+zr,yq, x^2p+(2x+az)zr,y^2q \}$   &$ x^2p+(2x+az)zr+bx+cz,  $  \\
  & & $ y^2q+dy$ \bsep{0.5ex} \\ \hline
 \tsep{0.5ex}17A$_2$ &    $ \{ p,q,xp+azr,yq+zr,x^2p+2axzr,y^2q+2yzr \}$  &$x^2p+2axzr+cx $,  \\
  &  &$y^2q+2yzr+dy $\bsep{0.5ex} \\ \hline
\tsep{0.5ex}17C & $ \{p,q,xp,yq,x^2p+l_0xzr,y^2q+p_0yzr, x^ly^pr\}$, &$x^2p+l_0xzr+cx $,
\\
 &$l\leq l_0,$ $b \leq b_0 $ & $y^2q+b_0yzr+dy$\bsep{0.5ex}\\
 \hline
 \tsep{0.5ex}17D & $
\{p,xp,x^2p,q,yp,y^2q,r,zr,z^2r\}$& $x^2p+ax$, $y^2q+cy$, $z^2r+dz$
\bsep{0.5ex} \\
  \hline
 \hline
\end{tabular}

}

\vspace{2mm}

Note here that $b$, $b_0$, $l$, $l_0$, $\tilde{n}$, $s$ and $t$ are
positive integers, $a$, $c$ and  $d$  are arbitrary constants.
Remark that Amaldi's Lie algebra 5A  is not a Lie algebra. Indeed,
for the space spanned by
$\{q,xq,\dots,x^tq+r,\dots,x^{t+i}q+\binom{t+i}{t}x^{i}r,\dots,x^sq+\binom{s}{t}x^{s-t}r,
p,yp,xp-tzr,yq+zr,x^2p+sxyq+(s-2t)xzr,
x^ly^{n-b}r,zr\} $  to be a Lie algebra,  the parameter $t$ needs to
be zero. We then get the Lie algebra 5A$^*$ given in the table.

\end{table}

\begin{table}[thp]\small {\centering

\caption{Type I quasi-exactly solvable Lie algebras of dif\/ferential operators and
their f\/ixed modules.}

\vspace{2mm}

\begin{tabular}{|l l l|}
  \hline
\tsep{0.6ex} & Quantization condition  & Fixed module \bsep{0.6ex}\\
 \hline \hline
\tsep{0.6ex} 1 &
$x^2p+xyq-nx$, $xyp+y^2q-ny$ &$ \{ x^iy^jg(z)|$ $i+j\leq n, $  $g(z)\in
G^{(i+j)} $\},\\
 & & for $G^{(l)}$ f\/inite with $G^{(l)} \subseteq G^{(l-1)}$
\bsep{0.6ex} \\ \hline
\tsep{0.6ex} 2&  $0$&  $ \{ x^iy^jg(z)|$ $i+j\leq n, $  $g(z)\in  G^{(i+j)}$ \},   \\
& & for $G^{(l)}$ f\/inite with $Z^iG^{(l)} \subseteq G^{(l-1)}$\bsep{0.6ex}
 \\ \hline
\tsep{0.6ex} 3 &  $0$  &$ \{ x^iy^jg(z)|$ $i+j\leq n, $  $g(z)\in G^{(i+j)} $  \},  \\
& & for $ G^{(l)}$ f\/inite, and $Z^iG^{(l)}\subseteq Z^iG^{(l-1)}$ \bsep{0.6ex}\\
\hline
\tsep{0.6ex} 4 &  $x^2p+xyq-nx$, $xyp+y^2q-ny$ &$ \{ x^iy^jz^{k}e^{\lambda_l z}|$  $i+j\leq n, $  $k\leq m_l,$ $ l=0,\dots,p \}  $     \bsep{0.6ex} \\ \hline
\tsep{0.6ex} 5&  $0$ &$ \{ x^iy^jz^{k}e^{\lambda_l z}|$ $i+j\leq n, $  $z^ke^{\lambda_l z}\in G^{(i+j)} \},$  \\
& & for $ G^{(l)}$  a f\/inite $z$-translation module,\\
& & and $Z^iG^{(l)}\subseteq G^{(l-1)}$ for $Z^i=z^ke^{\lambda_lz}p$
\bsep{0.6ex} \\ \hline
\tsep{0.6ex} 6 &  $0$  &$ \{ x^iy^jz^{k}e^{\lambda_l z} |$ $i+j\leq n, $  $k\leq m_l,$ $ l=0,\dots,p  \}$,  \\
& & for $ G^{(l)}$  a f\/inite $z$-translation module,\\
& & and $Z^iG^{(l)}\subseteq G^{(l-1)}$ for $Z^i=z^ke^{\lambda_lz}p$
\bsep{0.6ex} \\ \hline
\tsep{0.6ex} 7 &  $x^2p+xyq-nx$, $xyp+y^2q-ny$ &$ \{ x^iy^jz^k | $ $i+j\leq n, k\leq m    \}$\bsep{0.6ex}   \\
 \hline
\tsep{0.6ex} 8 &  $0$ &$ \{ x^iy^jz^k |$ $i+j\leq n, $   $l(i+j)+k\leq m    \}$
\bsep{0.6ex}\\ \hline
\tsep{0.6ex} 9 &  $0$  &$ \{ x^iy^jz^k |$ $i+j \leq n, $ $l(i+j)+k\leq
m $\}  \bsep{0.6ex}\\ \hline
\tsep{0.6ex}10 &  $x^2p+xyq-nx$, $xyp+y^2q-ny$,  &$ \{ x^iy^jz^k |$  $i+j\leq n, k\leq m   \} $   \\
& $z^2r-mz$&  \bsep{0.6ex}\\ \hline
\tsep{0.6ex}11 &  $z^2r+az(xp+yq)-mz$&$ \{ x^iy^jz^k |$ $i+j\leq n, $  $
a(i+j)+k\leq m    \}$   \bsep{0.6ex}\\ \hline
\tsep{0.6ex}12 &  $z^2r+az(xp+yq)-mz $  &$ \{
x^iy^jz^k |$ $i+j\leq n ,$ $ a(i+j)+k\leq m\}$    \bsep{0.6ex}\\
  \hline \hline
\end{tabular}

}

\vspace{2mm}

Note here that $m$ and $n$ are positive integers.

\bigskip

\small {\centering

\caption{Some of the type III quasi-exactly solvable Lie algebras of dif\/ferential
operators
and their f\/ixed modules.}

\vspace{2mm}

\begin{tabular}{|l l l|}
  \hline
\tsep{0.6ex} & Quantization condition  & Fixed module \bsep{0.6ex}\\   \hline \hline
\tsep{0.6ex}4A & $0$
& $ \{
x^iy^jz^k | \  i+sj+(s-t)k\leq n, j\leq m_y, k \leq m_z\}    $ \bsep{0.6ex}   \\
\hline
\tsep{0.6ex}4C &    $0$ &  $ \{ x^iy^jz^k | \  i+sj+(l_0+sn)k\leq n, k\leq b, j\leq b_k \}$  \\
& &   with   $b_{k-1}\geq b_k+n$    \bsep{0.6ex} \\
\hline
 \tsep{0.6ex}4D &    $z^2r-mz$   & \{$x^iy^jz^k | \  i+sj \leq n, k\leq m  \} $ \bsep{0.6ex}\\
 \hline
 \tsep{0.6ex}5A$^*$ &    $x^2p+sxyq+sr-nx$ &  $ \{ x^iy^jz^k | \
i+s(j+k)\leq n, j\leq m_y, k\leq m_z\}$     \bsep{0.6ex}\\ \hline
\tsep{0.6ex}5C &    $x^2p+sxyq+(l_0+s\tilde{n})xzr-nx$ &  $ \{ x^iy^jz^k | \  i+sj+(l_0+s\tilde{n})k\leq n,k\leq b, j\leq b_k \}$  \\
& &   with   $b_{k-1}\geq b_k+\tilde{n}$     \bsep{0.6ex}\\
\hline
\tsep{0.6ex}5D &    $x^2p+sxyq-nx$, $z^2-mz$&  $ \{ x^iy^jz^k | \  i+sj\leq n, k\leq m\}$   \bsep{0.6ex}\\  \hline
\tsep{0.6ex}7C &   $0$&   0\bsep{0.6ex}\\
\hline
 \tsep{0.6ex}17A$_1$ &   $ x^2p+(2x+az)zr-m_xx$, $y^2q-m_yy$&   $ \{ x^iy^j| \  i \leq m_x, j\leq m_y \}$   \bsep{0.6ex}\\ \hline
 \tsep{0.6ex}17A$_2$ &   $x^2p+2axzr-m_xx $, $ y^2q+2yzr-m_yy $&   $ \{ x^iy^jz^k| \  i+2ak \leq m_x,j+2k\leq m_y, k \leq m_z\}$   \bsep{0.6ex}\\ \hline
 \tsep{0.6ex}17C &   $x^2p+l_0xzr-m_xx $, $y^2q+b_0yzr-m_yy$&   $ \{ x^iy^jz^k| \  i+l_0k \leq m_x,j+b_0k\leq m_y, k\leq m_z \} $  \bsep{0.6ex}\\ \hline
  \tsep{0.6ex}17D &   $x^2p-m_xx$, $y^2q-m_yy $, $z^2r-m_zz$&   $ \{ x^iy^jz^k| \  i \leq m_x,$$j\leq m_y, k\leq m_z\} $   \bsep{0.6ex} \\ \hline
 \hline
\end{tabular}

}

\vspace{2mm}

Note here that $m$, $n$, $m_x$, $m_y$ and $m_z$ are positive
integers.

\end{table}

This set of principles is of great help in the determination of the
possible $\g$-modules $\nn$ for each Lie algebras of f\/irst order
dif\/ferential operators $\g$ described in Tables $1$ and $2$.
Depending on the elements contained in the Lie algebra studied, we
started our search of $\g$-module based on the general module given
in this guideline. Once again, the computations are tedious and it
would not be relevant to detail each of them.
 We will concentrate on the same Lie algebras as in the previous
 step of this classif\/ication problem,
  that is the type~I Lie algebras cases~$1$ and~$10$.

{\bf Type I, case 1.}
Since the Lie algebra contains the dif\/ferential operators $p$, $q$,
$xp$, and  $yp$, from the principle (\ref{test}), the most general
module $\nn$ will be spanned by functions of the form
$h=x^iy^jg^{i,j}(z)$, where $g^{i,j}(z)$ belongs to $G^{(i,j)}$. We
now consider the operator $T=x^2p+xyq+3xK(z)$ in the algebra $\g$
and its action on $h=x^ny^ag^{n,a}(z)$ a generator of $\nn$ with
maximal exponent in~$x$. Thus
\begin{gather*}T[h]=n
x^{n+1}y^ag^{n,a}(z)+ax^{n+1}y^ag^{n,a}(z)+3K(z)
x^{n+1}y^ag^{n,a}(z)\\
\phantom{T[h]}{} =(n+a+3K(z))x^{n+1}y^ag^{n,a}(z).
\end{gather*}
Since the
exponent in $x$ was taken to be maximal, this imposes that $K(z)$ is
indeed a constant~$K$ equal to $-\frac{n+a}{3}$. Symmetrically, by
considering $Q=xyp+y^2q+3yK(z)$ and $h=x^by^mg^{b,m}(z)$, a function
with maximal exponent in $y$, the following equality holds
$K=-\frac{b+m}{3}$. For this to be possible, we necessarily have
 $n+a=b+m$. Consequently, the dif\/ferential operators $xq$ and
$yp$ belong to the Lie algebra $\g$ and the  module $\nn$ is given
by the principle (\ref{test2}). Also note that the operators $xp$
and $yq$ force both $a$ and $b$ to be zero. Otherwise
$x^{n+1}y^{a-1}g^{n,a}(z)$ and $x^{b-1}y^{m+1}g^{b,m}(z)$ would be
in $\nn$, contradicting the maximality of $n$ and $m$.
 Thus $3K(z)=-n$ and the module
 \begin{equation}\label{module}
\nn=\{x^iy^jg^{i,j}(z) \ | \ i+j\leq n, \ g^{i,j}(z)\in G^{(i+j)}\},
\qquad \textrm{where}\quad  G^{(l)}\subseteq G^{(l-1)}
\end{equation}  is
f\/ixed by all the dif\/ferential operators in $\g$.
 Therefore it is possible conclude that the Lie algebra
\[ \g={\rm Span} \{p,q,xp,yq,xq,yp,x^2p+xyq-nx,xyp+y^2q-ny,1 \},\]
is quasi-exactly solvable with respect to the f\/inite dimensional $\g$-module~$\nn$.

{\bf Type I, case 10.}
Since the case $10$ Lie algebra  contains the case~$1$ Lie algebra,
its module~$\nn$ will be at most the module given in~(\ref{module}).
Observe f\/irst that the constant $c$ in the case $10$ Lie algebra
 has to be the negative integer $-n$. Furthermore, the
operator $r$ imposes  $\nn$ to be a $z$-translation module and the
operator $zr$ forces $G^{(l)}$ to be generated by monomials. Then,
for $z^m$ a monomial of maximal degree in $G^{(l)}$, the function
$h=x^iy^{l-i}z^m$ belongs to $\nn$. Since $z^2r+dz$ belongs to the
Lie algebra,
\begin{gather*}z^2r+dz[h] =mx^iy^{l-i}z^{m+1}+dx^iy^{l-i}z^{m+1}
=[m+d]x^{i}y^{l-i}z^{m+1},
\end{gather*} should belong to the
$\g$-module $\nn$. Thus, from the maximality of the degree in $z$,
 the constant $d$ has to be the negative integer
$-m$. Since the argument must hold for every set $G^{(l)}$, they
will all share the same monomial of maximal degree $m$. We can
therefore conclude that the Lie algebra
\[ \g={\rm Span} \{p,q,xp,yq,xq,yp,x^2p+xyq-nx,xyp+y^2q-ny,r,zr,z^2r-mz,1 \},\]
is quasi-exactly solvable with respect to the module
\[
\nn=\{x^iy^jz^k \ | \ i+j\leq n, k\leq m\}.\]

To summarize, a partial classif\/ication of quasi-exactly solvable Lie algebras of f\/irst
order dif\/ferential operators was accomplished in this section and
the description of these Lie algebras~$\g$, along with their
$\g$-modules, can be found in Tables $1$--$4$. In principle, it would
be possible to achieve a complete classif\/ication using similar
arguments.  However this gigantic work would require a colossal
amount of time. Nevertheless this partial classif\/ication is a good
starting point for seeking new quasi-exactly solvable Schr\"odinger operators in three
dimensions. In that scope, the next section is devoted to the
description of few new quasi-exactly solvable Schr\"odinger operators.

\section{New quasi-exactly solvable  Schr\"odinger operators\\ in three dimensions}

Recall that in the general classif\/ication problem, once the quasi-exactly solvable Lie
algebras of dif\/ferential operators $\g$ are determined, the next
step is to construct second order dif\/ferential operators $\hh$ that
are locally equivalent to Schr\"odinger operators. Given $\g$, one of
the Lie algebras of f\/irst order dif\/ferential operators obtained in
the previous section, we obtain a second order dif\/ferential operator
$\hh$ by letting
\begin{equation} \mathcal{H}=\sum_{a,b=1}^mC_{ab}T^aT^b+\sum_{a=1}^mC_aT^a+C_0,
 \qquad \textrm{where} \quad T^a\in \g,\end{equation}
 as illustrated previously. Then, one has to choose the coef\/f\/icients
 $C_{ab},$ $C_a$, $C_0$ in such a way that the closure conditions $d \omega=0$
 are satisf\/ied.
 Given $T^a=\xi^{ai}\partial_i+\eta^a$, the closure
conditions are the Frobenius compatibility conditions for the
overdetermined system
\[ \sum^m_{a,b=1}\xi^{ai}\left[
C_{ab}\left(\sum_{j=1}^n(\xi^{bj}\frac{\partial \alpha}{\partial
x^j}+\frac{\partial \xi^{bj}}{\partial
x^j})-2\eta^b\right)-C_a\right]=0,
\]
where $\alpha=\lambda+\frac{1}{2}\ln(g)$ and $\mu=e^\lambda$ is the
gauge factor. Finally, we need to bear in mind that the last step in
the classif\/ication is to verify that the operators are normalizable,
i.e.\ the functions in $\widetilde{\nn}$, the module obtained
after the gauge transformation, need to be square integrable. These
operators will therefore
have the property that part of their spectrum can be explicitly computed.

We have now in hand a large variety of generating quasi-exactly solvable Lie algebras
$\g$. The door is therefore wide open to the construction of
numerous new quasi-exactly solvable Schr\"odinger operators in three dimensions. However
the Schr\"odinger operators described in this paper are built only
from two of these new quasi-exactly solvable Lie algebras: the  type~III cases~17D and case~ 5A$^*$. The reader can therefore see that many more examples can be
constructed using this method together with the results of the
previous section.

\subsection[Type III, case 17D, ...]{Type III, case 17D, $\boldsymbol{(\mathfrak{sl(2)}\times \mathfrak{sl(2)}\times \mathfrak{sl(2)})}$}

The f\/irst two families of normalizable quasi-exactly solvable Schr\"odinger operators
displayed in this section are similar to the operator given in the
example (\ref{ex}). However, these two examples are more general.
Indeed, for these two operators, the type III  case 17D quasi-exactly solvable
Lie algebra $\g$ is spanned by the f\/irst order dif\/ferential
operators
\[
     p, \quad xp, \quad x^2p-m_xx, \quad q, \quad yq, \quad y^2q-m_yy, \quad r,  \quad zr, \quad
     z^2r-m_zz,
\] where $m_x$, $ m_y$ and $m_z$ are non negative integers and the module $\nn_{m_xm_ym_z}$ is generated by the $(m_x+1)(m_y+1)(m_z+1)$ monomials
\[ x^iy^jz^k \qquad \textrm{where} \quad 0\leq i \leq m_x, \quad j\leq m_y \quad \textrm{and}\quad  k\leq m_z.
\]

\subsubsection{First example}

 The f\/irst family of operators is constructed with the
following choice of coef\/f\/icients
\begin{gather*}
C_{ab}=\left(
  \begin{array}{ccccccccc}
    A & 0 & 1 & 0 & 0 &0& 0 & 0 & 1 \\
    0& B & 0 & 0 &0 & 0 & 0 & 1 & 0 \\
1  & 0 & C & 0 &0 & 0 & 1 &0   & 0 \\
    0 & 0 &0 & 2A & 0 & 0 & 0& 0 & 0 \\
    0 & 0 &0& 0 & 2B & 0 &0& 0 & 0 \\
    0 & 0 &0 & 0 & 0 & 2C & 0 & 0 & 0 \\
    0 &0 & 1 &0& 0 & 0 &A& 0 & 1 \\
0& 1 & 0 & 0 & 0 & 0 &0& B & 0 \\
    1 & 0 & 0 & 0 & 0 & 0 & 1 & 0 & C \\
  \end{array}
\right),
\\  C_c= [0, 0, 0, -2 Am_x, -2Bm_y, -2Cm_z, 0, 0, 0] \qquad \textrm{and} \qquad C_0 = -Am_x - Bm_y -
Cm_z.
\end{gather*}
 From the order two terms of these second order dif\/ferential
operators, the induced contravariant metric $g^{(ij)}$ is obtained
and reads as
\begin{equation}\label{metric2}
\left(
  \begin{array}{ccc}
    A(x^2+1)^2 & 0 & (x^2+1)(z^2+1) \\
    0 & By^4+2(B+1)y^2+B & 0 \\
    (x^2+1)(z^2+1) & 0 & C(z^2+1)
  \end{array}
\right). \end{equation} The determinant of the matrix is
$\widetilde{g}=(1-AC)(x^2+1)^2(By^4+2(B+1)y^2+B)(z^2+1)^2$ and one
easily verif\/ies, for $A$, $B$ and $C$  positive and  $AB>1$, that
the matrix is positive def\/inite on~$ \mathbb{R}^3$. The operator can
therefore be written as
 \[-2 \hh=\Delta+\vec{V}+U,\]
 where $\Delta$ is the
Laplace--Beltrami operator related to the metric (\ref{metric2}) and
where
\[ \vec{V}=-2(x^2+1)(Am_xx+m_zz) p-2m_y(By^3+By+y)q-2(z^2+1)(Cm_zz
+m_xx)r.\] From a direct computation, the closure conditions are
verif\/ied and the gauge factor
 required to gauge transform $\hh$ into a Schr\"odinger operator $\hh_0$ is given by
\[\mu
=(x^2+1)^{\frac{-m_x}{2}}(By^4+2(B+1)y^2+B)^{\frac{-m_y}{4}}(z^2+1)^{\frac{-m_z}{2}}.
\]
Once the transformation is performed, the equivalent operator reads
as
\[-2\hh_0=\Delta+U, \]
where the potential of the Schr\"odinger operator  $-\frac{U}{2}$ is
given by a
rational function in $y$.

 Note that the same three factors arise in both $\mu$ and $\widetilde{g}$. This
will simplify our computations while testing the square
integrability of the functions in  $\widetilde{\nn}$. Indeed, a
function in $\widetilde{\nn}$ is given by $h=\mu x^{i}y^{j}z^{k}$
where the exponents $i$, $j$ and $k$ are smaller or equal to $m_x$,
$m_y$ and $m_z$ respectively. Our aim is to show that the triple
integral
\[ \iiint_{\mathbb{R}^3}(\mu x^{i}y^{j}z^{k})^2\sqrt{g} dxdydz\]
is f\/inite. Obviously, it is suf\/f\/icient to show the convergence of
this integral for the monomials of maximal exponent.  We can
therefore focus on
\[\iiint_{\mathbb{R}^3}\frac{x^{2m_x}y^{2m_y}z^{2m_z}}
{(x^2+1)^{m_x+1}(By^4+2(B+1)y^2+B)^{\frac{m_y}{2}+\frac{1}{2}}
(z^2+1)^{m_z+1}} dxdydz.\]
Using Fubini's theorem, this triple
integral can be factored into the product of three integrals
\begin{gather*}
 \int_{-\infty}^{\infty}\frac{x^{2m_x}}{(x^2+1)^{m_x+1}}dx, \qquad \int_{-\infty}^{\infty}\frac{y^{2m_y}}
{(By^4+2(B+1)y^2+B)^{\frac{m_y}{2}+\frac{1}{2}}}dy, \qquad \textrm{and}\\
\int_{-\infty}^{\infty}\frac{z^{2m_z}}{(z^2+1)^{m_z+1}}dz ,
\end{gather*}
 each of which is easily shown to be convergent. We
therefore have in hand a normalizable quasi-exactly solvable Schr\"odinger operator
and it is feasible to determine explicitly part of its spectrum.

For instance, if we f\/ix $m_x=0$, $m_y=2$ and $m_z=1$, few
manipulations lead to the six eigenfunctions of the operator $\hh$
restricted to $\nn$. Indeed, with this choice of parameters, the
$\g$-module is
\[ \nn=\{1, y, y^2, z, yz, y^2z\},\]
and the transformation matrix to be diagonalized reads as
\[\left(
  \begin{array}{cccccc}
    -2\ \ & 0 & 2 & B & 0 & 0 \\
    0 & -4-2B & 0 & 0 & 0 & 0\\
    2B & 0 & -2\ \ & 0 & 0 & 0 \\
    0 & 0 & 0 &-2\ \ & 0 & 2B \\
    0 & 0 & 0 & 0 & -4-2B & 0 \\
    0 & 0 & 0 & 2B & 0 & -2
  \end{array}
\right).\] Once the diagonalization is performed,   three dif\/ferent
eigenvalues $\lambda_1=-4-2B $, $\lambda_2=-2-2B $, and
$\lambda_3=-2+2B $ are obtained,  each of them having multiplicity
two. The six eigenfunctions are respectively
\begin{gather*}
\psi_{1,1}=y, \qquad  \psi_{1,2}=yz,\qquad
\psi_{2,1}=1+y^2, \qquad
 \psi_{2,2}=-1+y^2,\\  \psi_{3,1}=-z+y^2z, \qquad
 \psi_{3,2}=z+y^2z.
\end{gather*}
 Consequently, we obtain three multiplicity two
 eigenvalues of the Schr\"odinger operator $\hh_0$: $\widetilde{\lambda_1}=1-B $, $\widetilde{\lambda_2}=1+B $
 and $\widetilde{\lambda_3}=2+B $,
 and the six scaled eigenfunctions are
\begin{gather*}
\widetilde{\psi_{1,1}}=\mu y, \qquad  \widetilde{\psi_{1,2}}=\mu yz,\qquad
 \widetilde{\psi_{2,1}}=\mu (1+y^2), \qquad
 \widetilde{\psi_{2,2}}=\mu (-1+y^2),\\
   \widetilde{\psi_{3,1}}=\mu (-z+y^2z), \qquad
 \widetilde{\psi_{3,2}}=\mu (z+y^2z).
\end{gather*}

As mentioned previously, the metric (\ref{metric2}) is positive
def\/inite on $\mathbb{R}^3$, hence Riemannian, and one can compute
that the Riemann curvature tensor is zero everywhere. The change of
variables that leads to a Cartesian coordinate system is given by
\[X=\arctan{x}, \qquad Y=\int{\frac{1}{\sqrt{By^4+2(B+1)y^2+B}}},\qquad
Z=\arctan{z},\] where we have some f\/lexibility on $B$ to adjust the
roots of the elliptic integral. Note that this operator is a good
illustration of the modif\/ied Turbiner's conjecture in three
dimensions proved in \cite{Moi}. Indeed, the generating Lie algebra
$\g$ is imprimitive and the leaves of the foliation are surfaces,
the Riemann curvature tensor is zero and, as expected, the potential
is separable since it depends on only one variable.

\subsubsection{Second example}
With the same representation of Lie algebra by f\/irst order
dif\/ferential operators but a dif\/ferent choice of coef\/f\/icients, one
constructs another family of second order dif\/ferential operators
$\hh$. Indeed, with
\begin{gather*}  C_{ab}=\left(
                     \begin{array}{ccccccccc}
                       A& 0& 1& 0& 0& 0& 0& 0 & \lambda \\
                       0& D& 0& 0& 0& 0& 0 & \beta & 0 \\
                       1& 0& B& 0& 0& 0& 1& 0& 0 \\
                       0& 0& 0& 2 A& 0& 0& 0& 0& 0 \\
                       0& 0& 0& 0& \beta D + C& 0& 0& 0& 0 \\
                       0& 0& 0& 0& 0& 2 \lambda B& 0& 0& 0 \\
                       0& 0& 1& 0& 0& 0& A& 0& \lambda\\
                       0& \beta& 0& 0& 0& 0& 0& \beta C& 0\\
                       \lambda& 0& 0& 0& 0& 0& \lambda& 0& \lambda^2
                       B
                     \end{array}
                   \right),\\
 C_c=[0, 0, 0, -2Am_x, -\beta D (1+2m_y)+C,
-2\lambda Bm_z, 0, 0, 0 ]  \qquad \textrm{and}\\
C_0=-Am_x-Bm_y-Cm_z,
\end{gather*} a family of operators $\hh$  is obtained and
one easily verif\/ies that all these operators are equivalent to
Schr\"odinger operators $\hh_0$. Note that this family of operators
is slightly more general than the family obtained in the f\/irst
example. However, some of the details are lengthy and are omitted
for brevity sake. The induced contravariant metric $g^{(ij)}$ is
given by:
\begin{equation}\label{metric3}
\left(
  \begin{array}{ccc}
    A(x^2+1)^2 & 0 & (x^2+1)(\lambda z^2+1)\\
    0 &  \beta C y^4+y^2(2 \beta+\beta D+C)+D & 0 \\
    (x^2+1)(\lambda z^2+1) & 0 & B(\lambda z^2+z))
  \end{array}
\right),
\end{equation}
and it is positive def\/inite on $\mathbb{R}^3$ provided $A$, $B$,
$C$, $D$ and $\beta$ are positive and $AB>1$. Its determinant is
$\widetilde{g}=(AB-1)(x^2+1)^2(\beta Cy^4+2\beta y^2+\beta
Dy^2+y^2C+D)(\lambda z^2+1)^2$,
 and the gauge factor required is the product of three
 functions in $x$, $y$ and $z$ respectively. After the
gauge transformation the new potential is again a rational function
involving only the $y$ variable and the Riemann curvature tensor is
null again. The following change of variables leads to Cartesian
coordinates
\begin{gather*}
X=\arctan{x}, \qquad Y=\int{\frac{1}{\sqrt{\beta C y^4+2 \beta y^2+2 \beta D y^2+y^2c+D}}dy}, \\ Z=\frac{\arctan{\sqrt{\lambda}z}}{\sqrt{\lambda}},
\end{gather*}
where we have some f\/lexibility on $\beta$, $C$ and $D$ to adjust
the roots of the elliptic integral Y.

However, we do not know if these operators are all normalizable.
But, if we f\/ix $C=\beta D$, the gauge transformation simplif\/ies and
becomes, once again, very similar to the determinant of the metric
(\ref{metric3}). Indeed
\[\mu=(x^2+1)^{-\frac{m_x}{2}}(\beta^2Dy^4+2\beta y^2(1+D)+D)^{-\frac{m_y}{4}}(\lambda z^2+1)^{-\frac{m_z}{2}}\]
and one verif\/ies, the exact same way as in the previous example,
that the functions in $\widetilde{\nn}$ are square integrable.
Therefore, for any choice of integers $m_x$, $m_y$, and $m_z$, one
would obtain $(m_x+1)(m_y+1)(m_z+1)$ eigenfunctions in the spectrum
of the Schr\"odinger operator $\hh_0$.

For instance, if we f\/ix $\lambda=1$, $\beta=5$  and the three
parameters $m_x$, $m_y$ and $m_z$ to be $1$, one gets two
 eigenvalues, $-3$ and $-7$ of multiplicity four, and the following eight
eigenfunctions
\begin{gather*} \widetilde{\psi_{-7,1}}= \mu (-1 + x z), \qquad \widetilde{\psi_{-7,2}}=\mu (y - x y z),\\ \widetilde{\psi_{-7,3}}=\mu (x y + y z),\qquad \widetilde{\psi_{-7,4}}=\mu (x + z),\\
\widetilde{ \psi_{-3,1}}= \mu (y + x y z),\qquad \widetilde{\psi_{-3,2}}= \mu (-x + z),\\  \widetilde{\psi_{-3,3}}=\mu (-x y + y z),\qquad \widetilde{\psi_{-3,4}}=\mu (1 + x z).
\end{gather*}
Note that the nodal surfaces can described easily in this coordinate
system. Indeed, since $\mu$ is always positive, the nodal surfaces
are simply the zero loci of polynomials. For these eight
eigenfunctions, the surfaces are given by the zeros of degree two
factorizable polynomials and one easily gets the following pictures.

\centerline{\includegraphics[width=45mm]{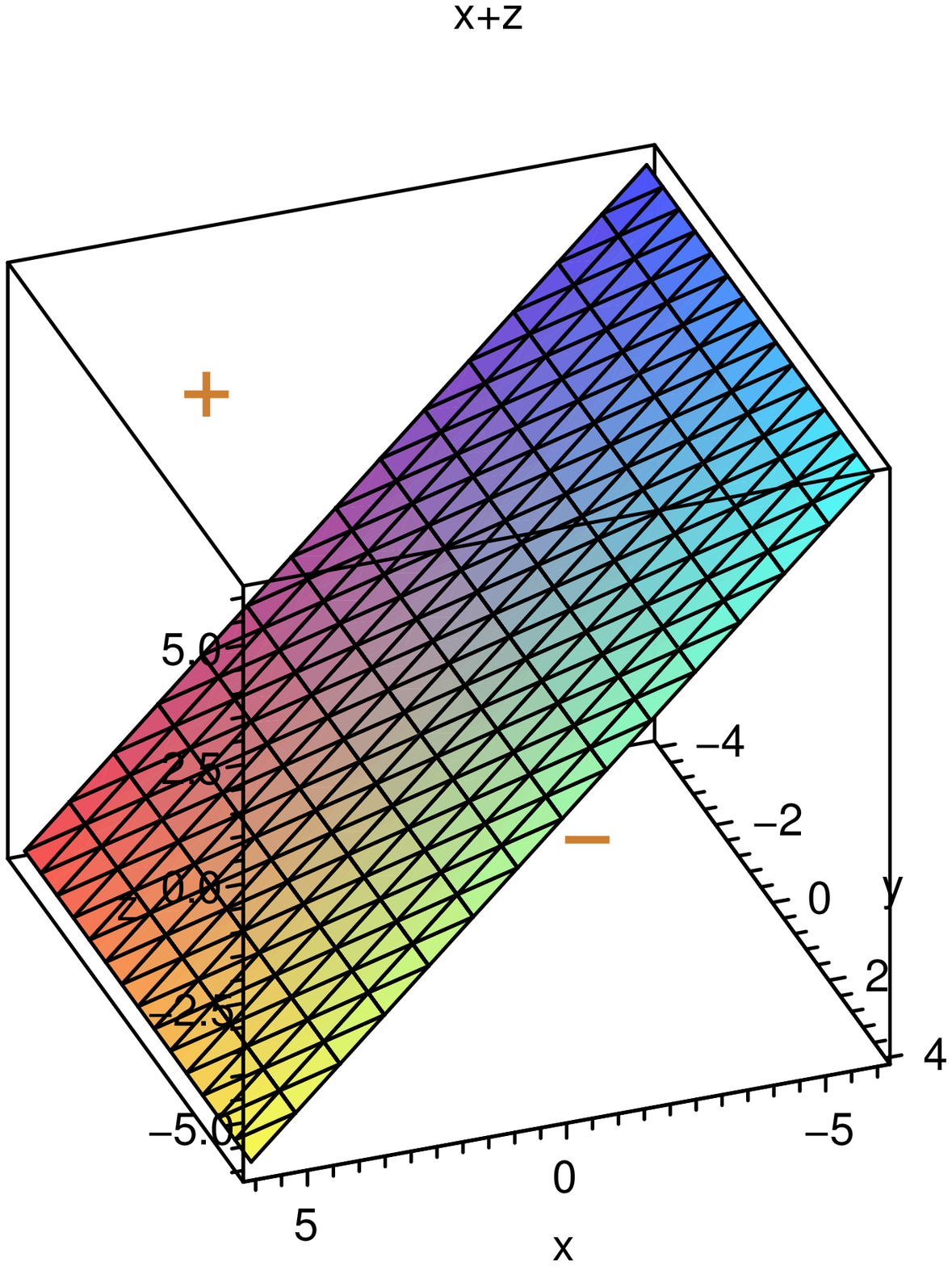} \qquad \qquad \includegraphics[width=45mm]{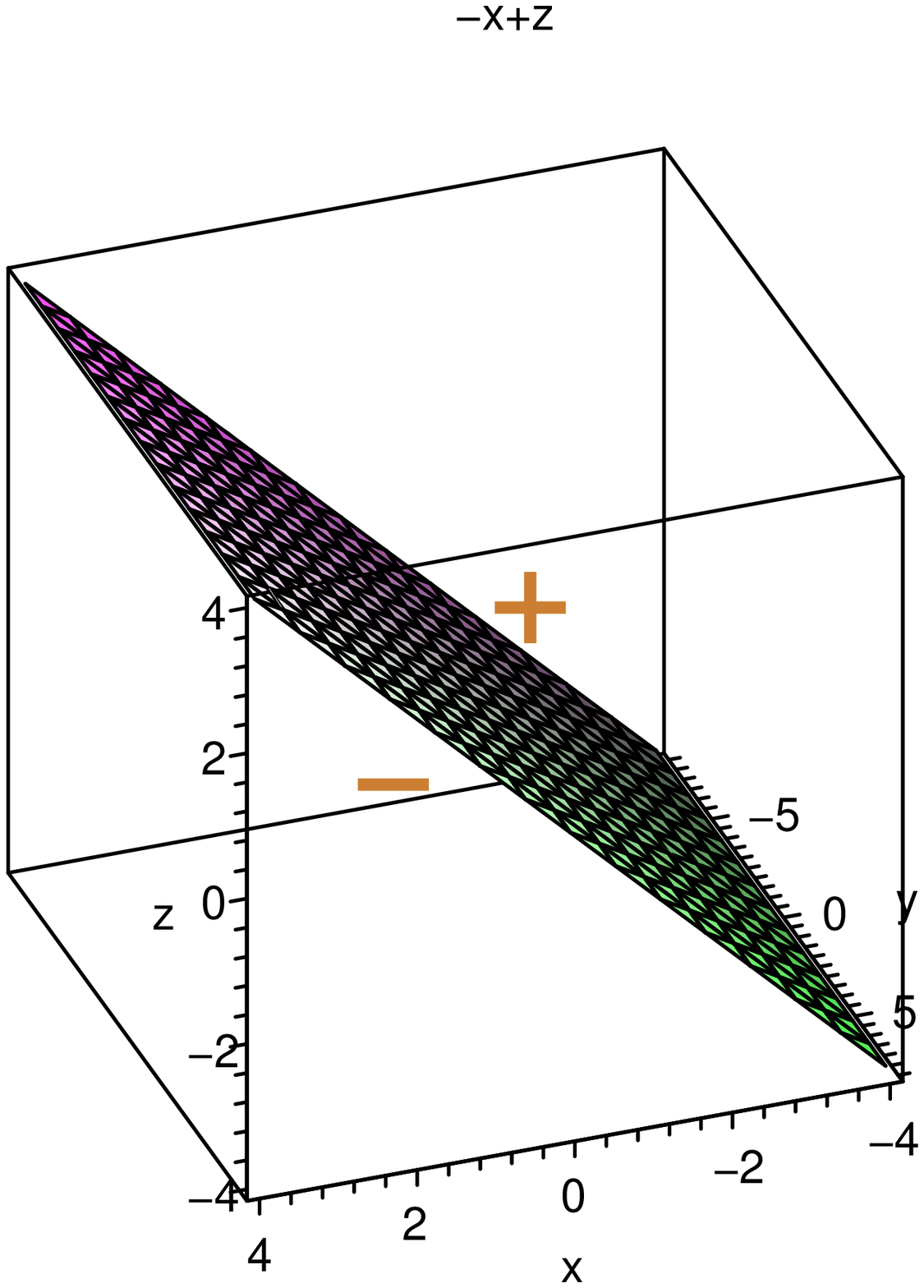}}

\bigskip

\centerline{\includegraphics[width=45mm]{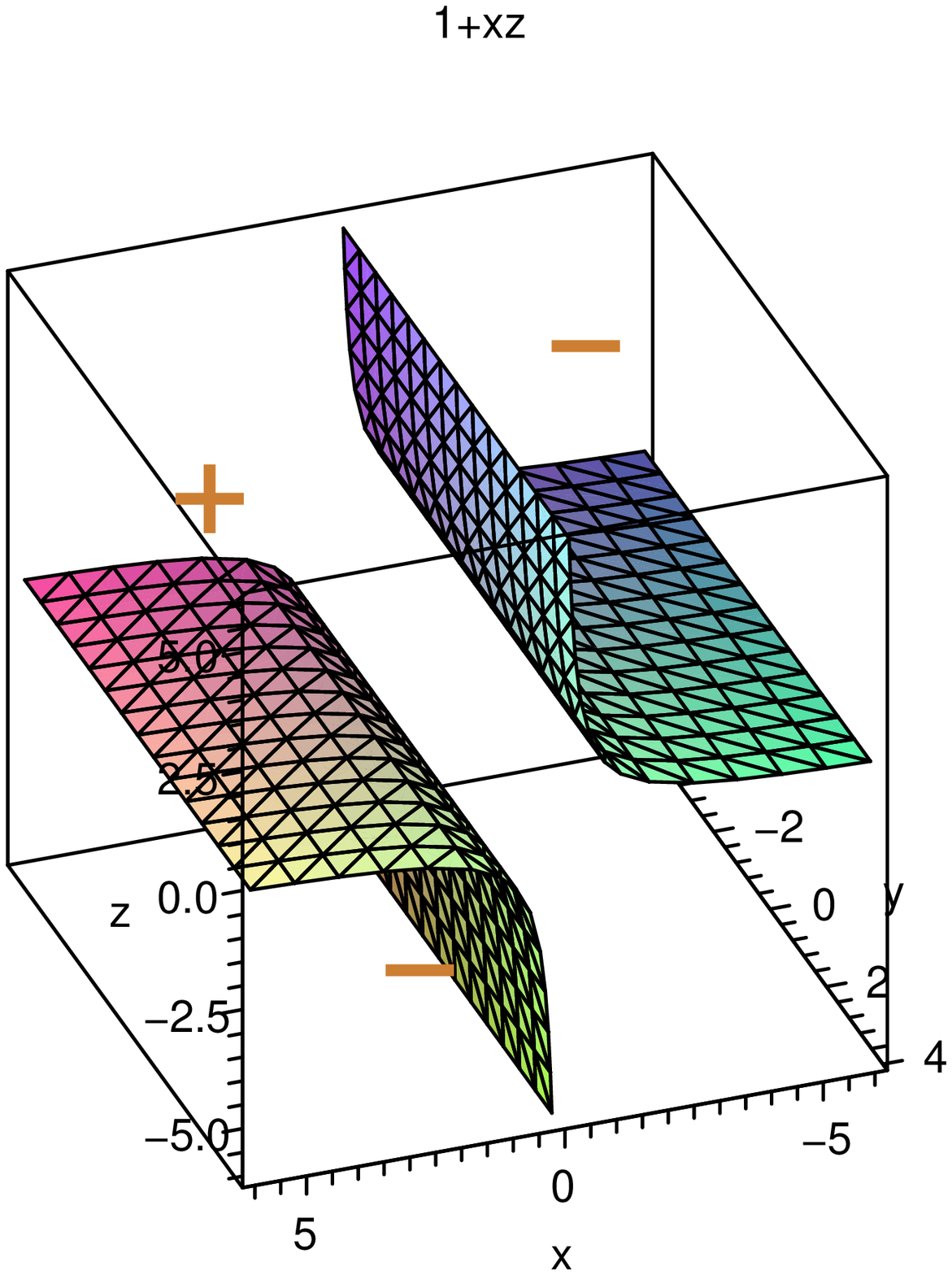} \qquad \qquad \includegraphics[width=45mm]{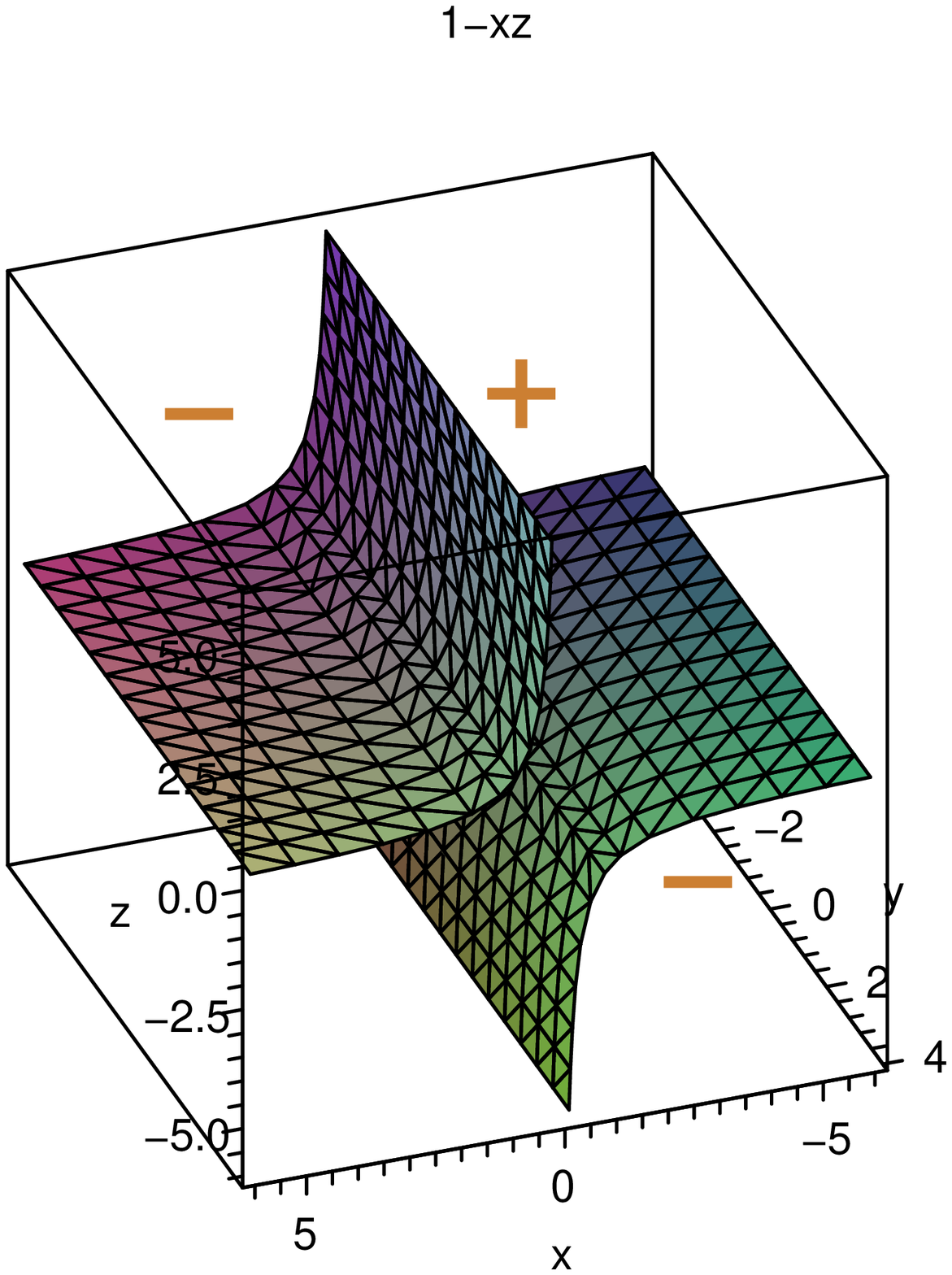}}

\bigskip

\centerline{\includegraphics[width=45mm]{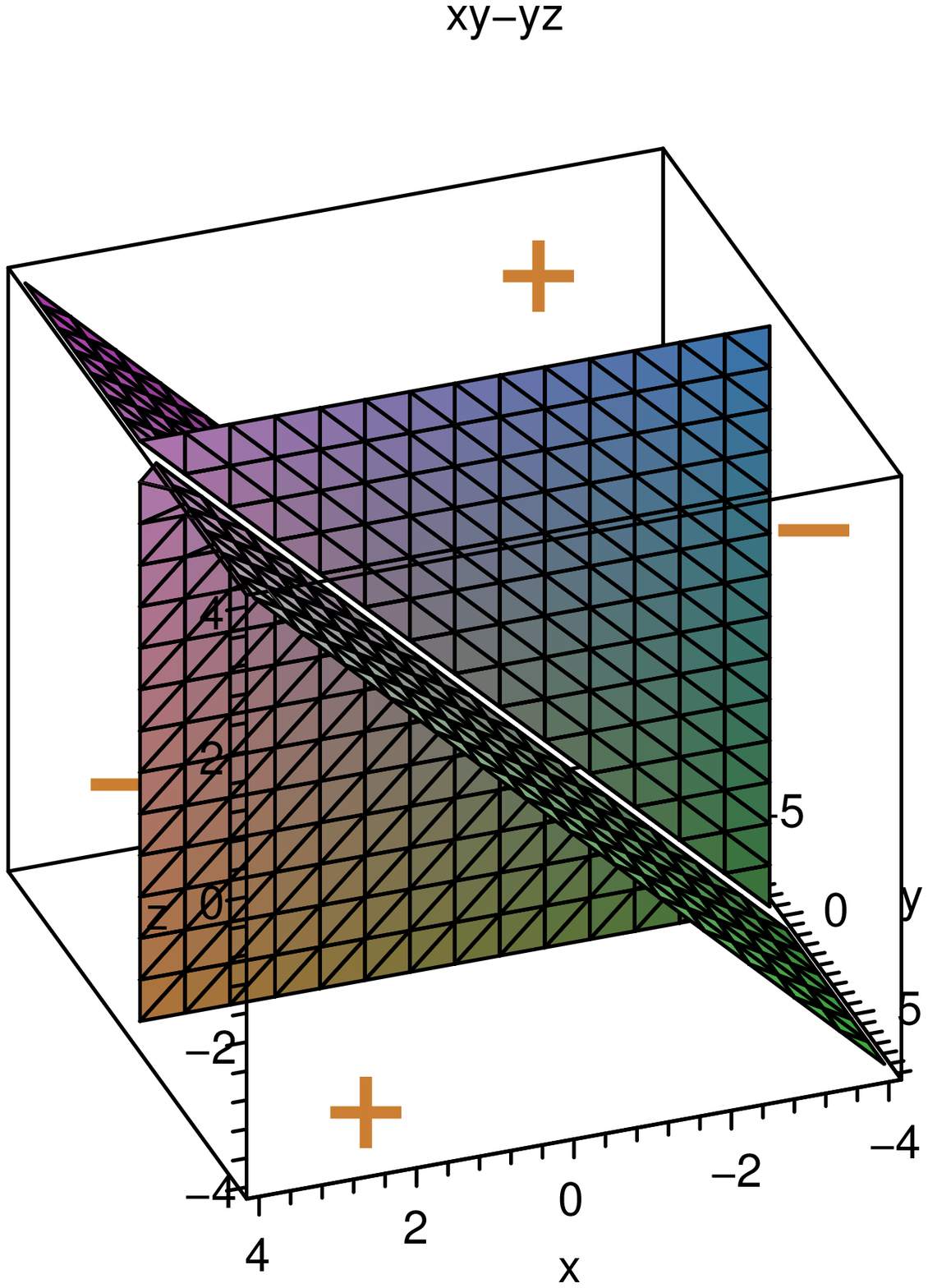} \qquad \qquad \includegraphics[width=45mm]{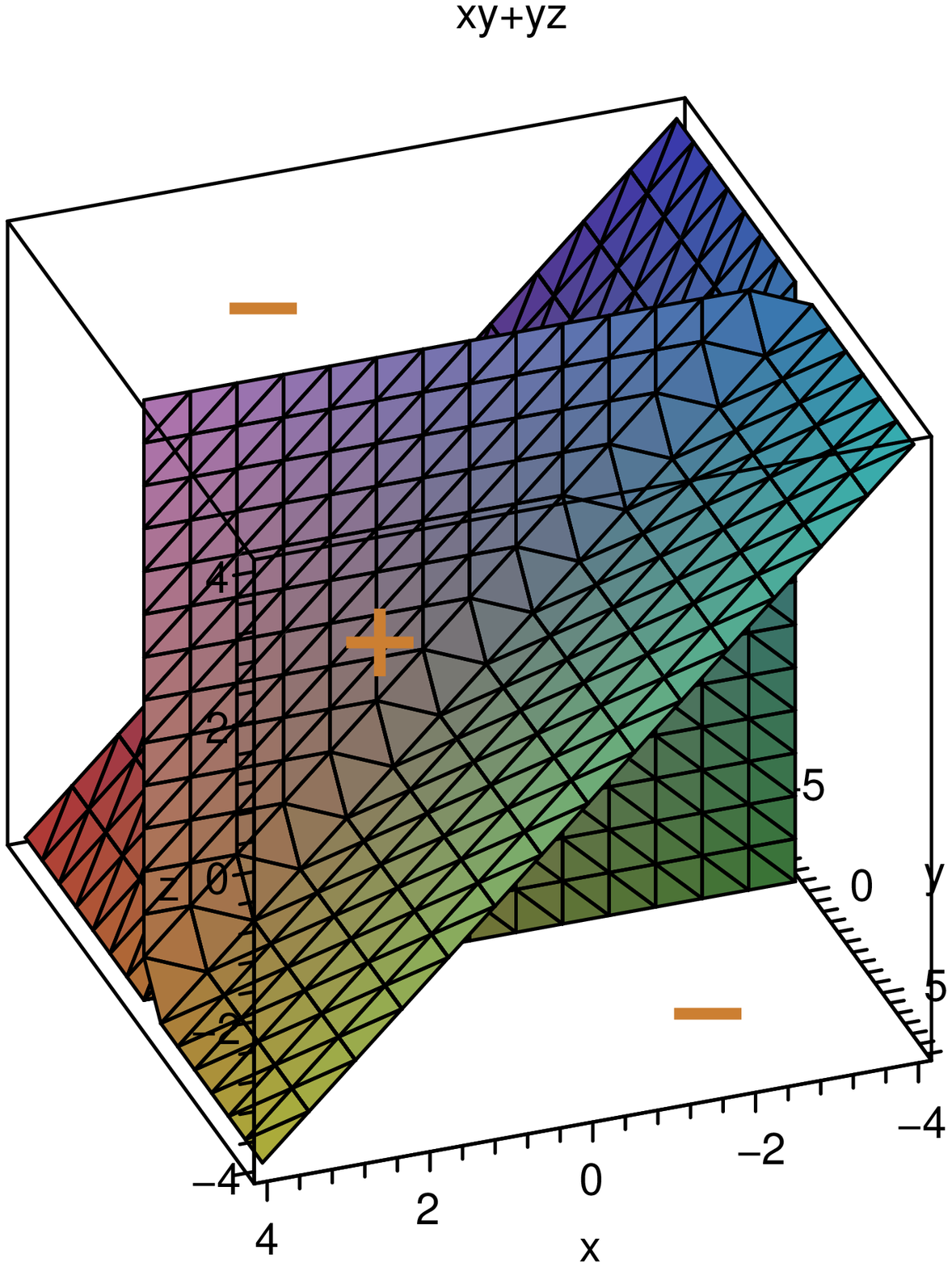}}

\bigskip

\centerline{\includegraphics[width=45mm]{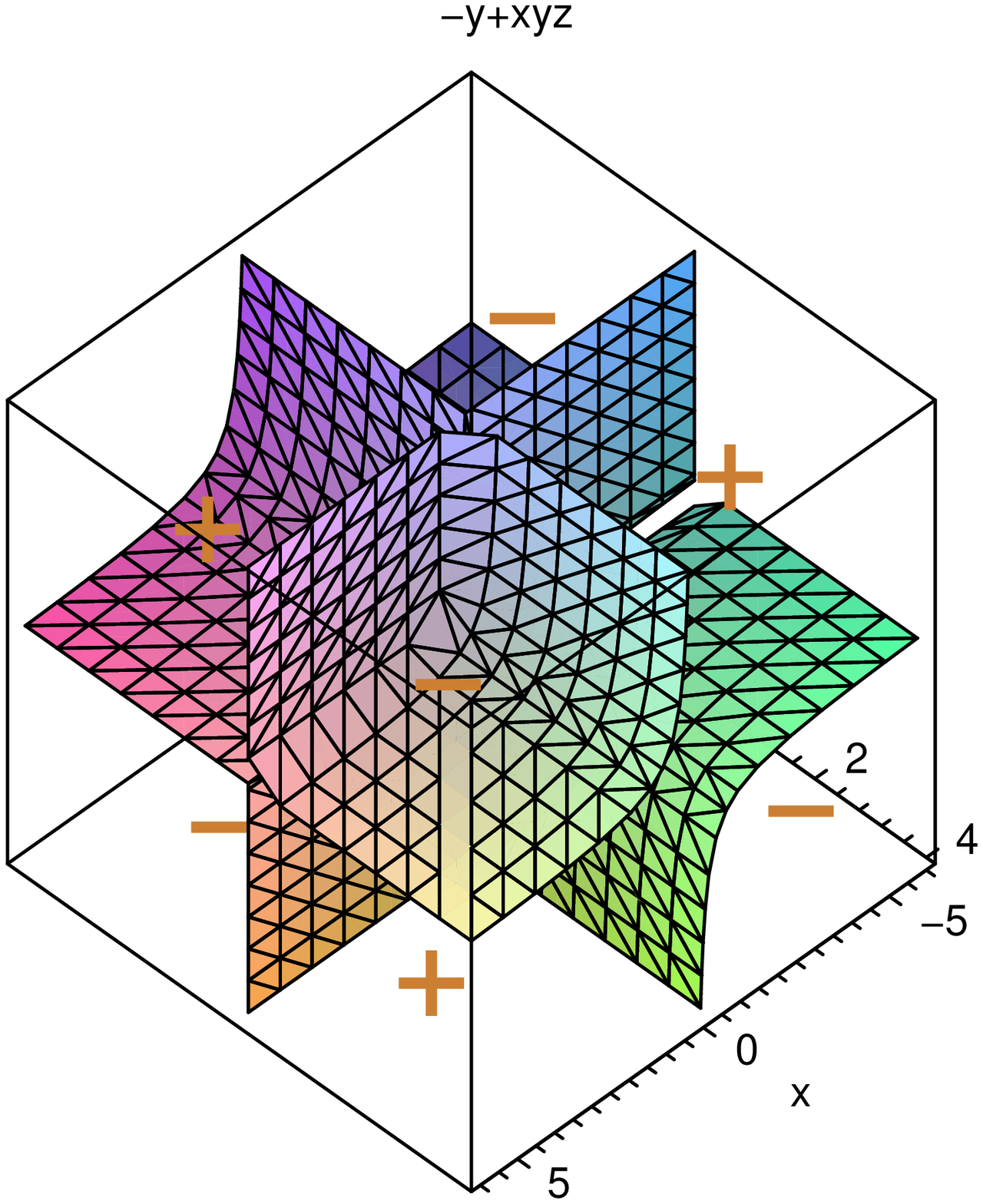} \qquad \qquad \includegraphics[width=45mm]{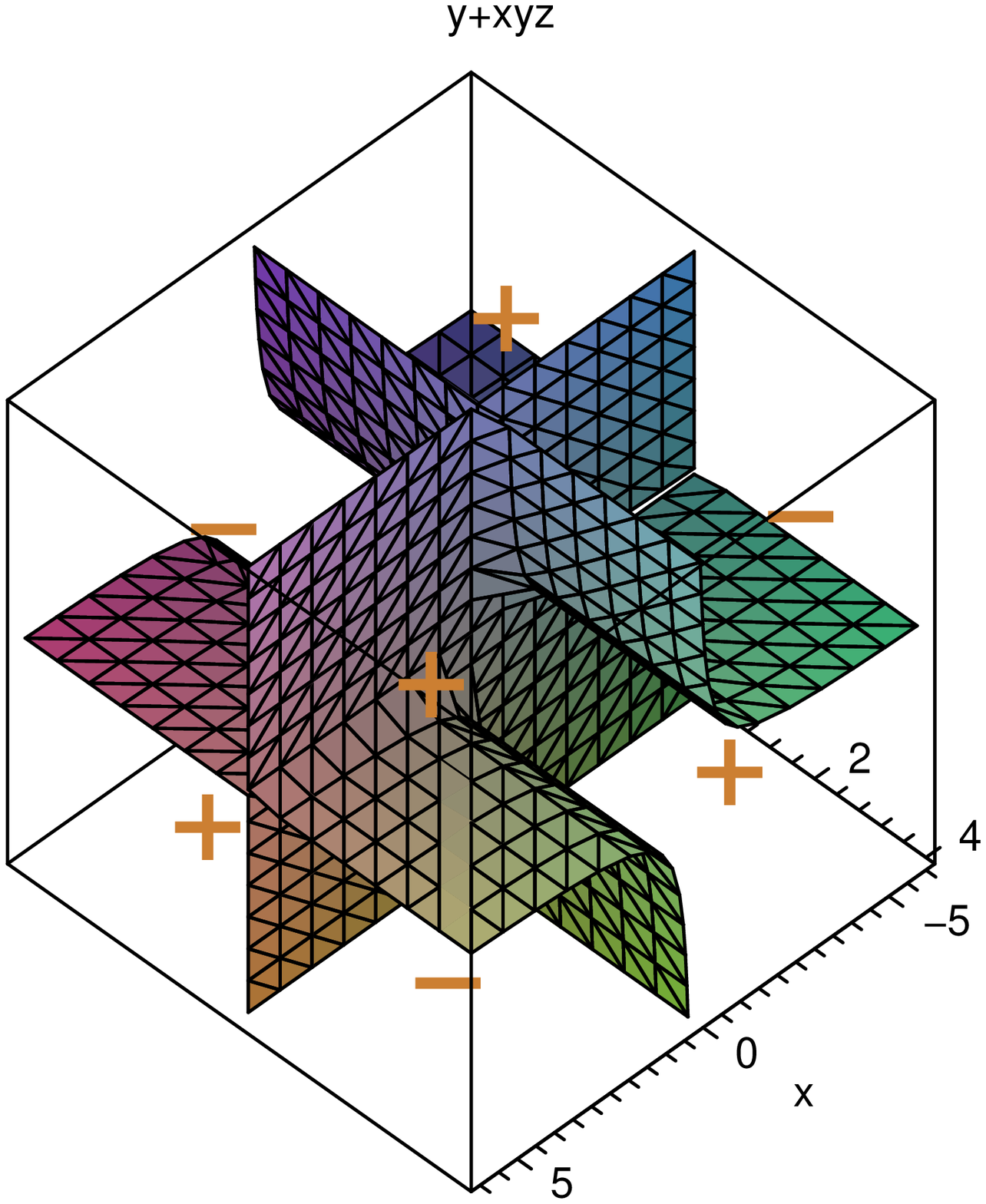}}

\subsection[Type III, case 5A$^*$, ...]{Type III, case 5A$\boldsymbol{^*}$, $\boldsymbol{\mathfrak{sl(2)}\ltimes\mathbb{C}^{s+1}}$}

For the last example, we consider the type III case 5A$^*$  quasi-exactly solvable
Lie algebra and we f\/ix the parameter $s$ to be one. This Lie algebra
$\g$ is therefore spanned by the following six f\/irst order
dif\/ferential operators
 \[ p,\quad q+r,xq+xr,\quad  xp,\quad yq+zr,\quad \textrm{and}\quad x^2p+xyq+xzr-nx , \]
and from the Table $4$, the $\g$-module of function is given by
 \[\nn=\{ x^iy^jz^k  | \ i+j+k\leq n, j\leq m_y, k\leq m_z \},\]
 where $n$, $m_y$ and $m_z$ are non-negative integers.
  A family of Schr\"odinger operators on $\mathbb{R}^3  \backslash  \{x=y\}$ is obtained from
 the following choice of coef\/f\/icients,
\begin{gather*}
 C_{ab}=\left(
   \begin{array}{cccccc}
     A & 0 & 0 & 0 & 0 & 0 \\
     0 & B & 0 & 0 & 0 & 0\\
     0 & 0 & C & 0 & 0 & 0 \\
     0 & 0 & 0 & 0 & 0 & 0\\
     0 & 0 & 0 & 0 & D & 0 \\
    0& 0 & 0 & 0 & 0 & 0
   \end{array}
 \right),
\\
 C_c=[
     0 , 0 , C , 0 , -2(1+m)D , 0 ], \qquad \textrm{and} \qquad C_0=(1+m)^2D,
\end{gather*}
where the parameters $A$, $B$, $C$, and $D$ are positive. The
induced contravariant metric is given by
\begin{equation}\label{metric5a}g^{(ij)}=\left(
  \begin{array}{ccc}
    Cx^2+A & 0 & 0 \\
    0 & Dy^2+B & Dyz+B \\
    0 & Dyz+B & Dz^2+B
  \end{array}
\right), \end{equation} its determinant is
$\widetilde{g}=BD(Cx^2+A)(y-z)^2$ and the metric is positive
def\/inite on $\mathbb{R}^3\backslash \{x=y\}$. Before performing the
gauge transformation the operator reads as
\begin{gather*}-2\hh =
\Delta+(2Cx-Cmx)p+(-Dy-2Dmy)q+(-Dz-2Dmz)r\\
\phantom{-2\hh =}{} -1/2Cm+1/4Cm^2+(1+m)^2D,\end{gather*}
 and one easily verif\/ies
that the operator respects the closure condition. The gauge factor
required to obtain a Schr\"odinger operator is
\[
\mu=(Cx^2+A)^{\frac{1-m}4}(y-z)^{\frac{-2m-3}{2}},
\] and once again,
contains the same factors as the determinant of the covariant
metric. Finally, after the gauge transformation, the Schr\"odinger
operator reads as,
\[
-2 \hh_0=\Delta+U,
\]
where $U$ depends on the three variables. Although, it is not known
if the functions in $\widetilde{\nn}$ are square integrable on the
domain
$\mathbb{R}^3 \backslash \{x=y\}$.

Note that for this example, the scalar curvature is constant and
depends on the parameter~$D$ while the Riemann curvature tensor is
equal to
 \[\frac{-1}{B(y-z)^2}dy dz dy
dz.\] However, the potential does not seem to be separable.

\subsection*{Acknowledgements}

The research is supported by a NSERC Grant $\#$RGPIN $105490-2004$
and a McGill Graduate Studies Fellowship. I would like to thank Niky
Kamran, for all the encouragement
  and the precious advice he gave me.

\pdfbookmark[1]{References}{ref}
\LastPageEnding

\end{document}